\documentclass[a4paper,12pt]{amsart}
\usepackage[T1]{fontenc}
\usepackage[english]{babel}
\usepackage[left=1in,right=1in,top=1.1in,bottom=1.1in]{geometry}

\allowdisplaybreaks

\usepackage{mathtools,amssymb,mathrsfs,amscd,comment}
\usepackage{times}
\usepackage[nopatch=footnote]{microtype}
\usepackage{enumitem}
\usepackage{tikz-cd}

\usepackage{xcolor}

\usepackage[pagebackref,linktocpage]{hyperref}

\hypersetup{
	hypertexnames=false,
	pdftitle={Height growth of rational self-maps},
	pdfauthor={Hu--Song},
	colorlinks=true,
	linkcolor=red,
	citecolor=blue,
	urlcolor=blue
}
\usepackage[nameinlink,noabbrev]{cleveref}

\newtheorem{theorem}{Theorem}[section]
\crefname{theorem}{Theorem}{Theorems}
\newtheorem{lemma}[theorem]{Lemma}
\crefname{lemma}{Lemma}{Lemmas}

\crefname{proposition}{Proposition}{Propositions}

\crefname{prop}{Proposition}{Propositions}
\newtheorem{corollary}[theorem]{Corollary}
\crefname{corollary}{Corollary}{Corollaries}

\crefname{cor}{Corollary}{Corollaries}

\crefname{conjecture}{Conjecture}{Conjectures}

\crefname{conj}{Conjecture}{Conjectures}
\newtheorem*{conj*}{Conjecture}
\crefname{conj}{Conjecture}{Conjectures}

\crefname{thmx}{Theorem}{Theorems}

\theoremstyle{definition}
\newtheorem{definition}[theorem]{Definition}
\crefname{definition}{Definition}{Definitions}

\crefname{defn}{Definition}{Definitions}

\crefname{example}{Example}{Examples}

\crefname{notation}{Notation}{Notation}
\newtheorem*{notation*}{Notation}
\crefname{notation}{Notation}{Notation}
\newtheorem*{convention*}{Convention}
\crefname{convention}{Convention}{Convention}

\crefname{problem}{Problem}{Problems}

\crefname{question}{Question}{Questions}

\crefname{condition}{Condition}{Conditions}

\crefname{assumption}{Assumption}{Assumptions}
\newtheorem{settinginner}{Setting}
\newenvironment{setting}[1]
  {\renewcommand{\thesettinginner}{(#1)}\begin{settinginner}}
  {\end{settinginner}}
\crefname{settinginner}{Setting}{Settings}
\Crefname{settinginner}{Setting}{Settings}

\theoremstyle{remark}

\crefname{rmk}{Remark}{Remarks}
\newtheorem*{rmk*}{Remark}
\crefname{rmk}{Remark}{Remarks}
\newtheorem{remark}[theorem]{Remark}
\crefname{remark}{Remark}{Remarks}

\crefname{fact}{Fact}{Facts}
\newtheorem{claim}[theorem]{Claim}
\crefname{claim}{Claim}{Claims}
\newtheorem*{claim*}{Claim}
\crefname{claim}{Claim}{Claims}

\crefname{step}{Step}{Steps}

\crefname{case}{Case}{Cases}

\numberwithin{equation}{section}

\DeclareMathOperator{\Pic}{Pic}

\newcommand{\Ebar}{\overline{E}}
\newcommand{\Hbar}{\overline{H}}
\newcommand{\Kbar}{\overline{K}}
\newcommand{\Lbar}{\overline{L}}
\newcommand{\Mbar}{\overline{M}}
\newcommand{\Nbar}{\overline{N}}

\newcommand{\ratmap}{\dashrightarrow}

\newcommand{\bC}{\mathbf{C}}

\newcommand{\bk}{\mathbf{k}}

\newcommand{\bP}{\mathbf{P}}
\newcommand{\bQ}{\mathbf{Q}}
\newcommand{\bR}{\mathbf{R}}

\newcommand{\bZ}{\mathbf{Z}}

\newcommand{\cH}{\mathcal{H}}

\newcommand{\cN}{\mathcal{N}}
\newcommand{\cO}{\mathcal{O}}
\newcommand{\cP}{\mathcal{P}}

\newcommand{\cX}{\mathcal{X}}

\newcommand{\NT}{\mathrm{NT}}

\newcommand{\sB}{\mathscr{B}}

\newcommand\bqed{\hfill $\blacksquare$}

\providecommand{\MR}[1]{}
\renewcommand{\MR}[1]{\href{https://mathscinet.ams.org/mathscinet-getitem?mr=#1}{MR#1}}

\newcommand{\FH}[1]{{\color{red} H: #1}}

\title[Height growth of rational self-maps]{Height growth of rational self-maps}

\author{Fei Hu}
\address{School of Mathematics, Nanjing University, Nanjing, China}
\email{fhu@nju.edu.cn}

\author{Jiarui Song}
\address{School of Mathematical Sciences, Peking University, Beijing, China}
\email{soyo999@pku.edu.cn}
\date{\today}

\subjclass[2020]{37P15, 11G50, 14G40, 14E05}

\keywords{Height growth, arithmetic dynamical degree, dynamical degree, dominant rational map, arithmetic Hodge index theorem, Hodge--Lorentz package}

\thanks{The first author was supported by NSFC Grant \#12371045.}

\begin{document}

\begin{abstract}
Let $K$ be a number field or a function field of a curve.
Let $f$ be a dominant rational self-map of a normal projective variety $X$ of dimension $d$ over $K$, and let $h_H\geq1$ be an ample Weil height.
Let $\lambda_1(f)$ denote the first dynamical degree of $f$.
We show that there is a constant $C_{f,h_H}>0$ such that for any $n\geq1$ and any point $P\in X(\overline{K})$ with well-defined forward orbit,
\[
h_H(f^n(P))\leq C_{f,h_H} h_H(P)
\begin{cases}
n^{2(d-1)}\lambda_1(f)^n, &\lambda_1(f)>1,\\
n^{2d}, &\lambda_1(f)=1.
\end{cases}
\]
Over number fields, the proof combines a new uniform upper bound for arithmetic degrees with a uniform comparison of point heights.
Over function fields, it relies on the corresponding uniform upper bound for geometric degrees on a projective model.
\end{abstract}

\maketitle


\section{Introduction}\label{sec:intro}

We work throughout in either of the following settings.

\begin{setting}{N}\label{setting:number-field}
Let $K$ be a number field.
Let $X$ be a normal projective variety of dimension $d$ and let $f\colon X\ratmap X$ be a dominant rational self-map, over $K$.
\end{setting}

\begin{setting}{F}\label{setting:function-field}
Let $B$ be a smooth projective curve over an algebraically closed field $\bk$ of arbitrary characteristic, and let $K\coloneqq\bk(B)$ be its function field.
Let $X$ be a normal projective variety of dimension $d$ and let $f\colon X\ratmap X$ be a dominant rational self-map, over $K$.
\end{setting}



Let $h_H$ be a logarithmic Weil height associated with an ample line bundle $H$ on $X$ and set $h_H^+\coloneqq\max\{1,h_H\}$.
We refer to \cite{GTM201,BG06} for Weil height theory.
Denote by $X_f(\Kbar)$ the set of points with well-defined forward orbit.

Matsuzawa \cite{Matsuzawa20a} first established the following uniform height estimate for smooth projective varieties, which was later extended by Xie \cite{Xie26-DML} to arbitrary projective varieties (see also \cite[Proof of Proposition~3.11]{JSXZ24}).
The estimate was originally stated by Kawaguchi and Silverman \cite{KS16b} (a weaker version has already appeared in \cite[Proposition~12]{Silverman14} for projective spaces).
However, there is a delicate gap in their proof concerning the dependence of certain constants on the iterate time pointed out by Kaoru Sano.

\begin{theorem}[{see \cite[Theorem~6]{Xie26-DML}, cf.~\cite[Theorem~1]{KS16b} and \cite[Theorem~1.4]{Matsuzawa20a}}]
Assume that either \cref{setting:number-field} or \cref{setting:function-field} holds.
Then for every $\epsilon>0$, there exists a constant $C_{f,h_H,\epsilon}>0$, such that for all $P\in X_f(\Kbar)$ and $n\geq1$,
\begin{equation}\label{eq:KS-M}
h_H^+(f^n(P))\leq C_{f,h_H,\epsilon}\bigl(\lambda_1(f)+\epsilon\bigr)^nh_H^+(P).
\end{equation}
\end{theorem}

Here, for $0\leq i\leq d$, the $i$-th dynamical degree $\lambda_i(f)$ is a basic geometric invariant of the map, capturing the exponential growth rate of the $i$-th degrees $\deg_{i,H}(f^n)$ of the iterates $f^n$.
We refer to \cite{DS05a,Truong20,Dang20,Xie-SC} for their basic properties.
For surjective endomorphisms, Matsuzawa also obtained exponential-polynomial upper bounds with polynomial exponents depending on the Picard number \cite[Theorem~1.7]{Matsuzawa20a}.
See \cite{Sano18,Shibata19} for more on height growth of endomorphisms.

Recently, joint with Jiang \cite{HJ-degree-growth}, the first author developed a Hodge--Lorentz framework to study degree growth of endomorphisms, which was subsequently extended to study rational and meromorphic self-maps.
In particular, an exponential-polynomial upper bound for degree sequences has been established \cite[Theorem~1.1]{Hu-degree-growth}.
As an arithmetic application of this framework, the main purpose of this paper is to strengthen \eqref{eq:KS-M} by replacing the exponential loss with an explicit polynomial factor whose degree depends only on the dimension and the dynamical degree sequence $(\lambda_i(f))_i$.

For a finite real sequence $a=(a_0,\dots,a_m)$, an index interval $[r,s]$, with $0\leq r<s\leq m$, is \emph{affine} for $a$ if $2a_i-a_{i-1}-a_{i+1}=0$ for every integer $r<i<s$.
It is \emph{maximal affine} if it is not properly contained in any larger affine index interval.

\begin{theorem}\label{thm:ht-growth-upper-bound}
Assume that either \cref{setting:number-field} or \cref{setting:function-field} holds.
Then there exists a constant $C_{f,h_H}>0$ such that the following statements hold for every $P\in X_f(\Kbar)$ and $n\geq1$.
\begin{enumerate}[label=\textnormal{(\arabic*)}]
\item If $\lambda_1(f)>1$ and $[0,s]$ is a maximal affine interval of $(\log\lambda_i(f))_{i=0}^d$, then
\begin{equation}
h_H^+(f^n(P))\leq C_{f,h_H} n^{2(s-1)}\lambda_1(f)^n h_H^+(P).
\end{equation}
In particular,
\[
h_H^+(f^n(P))\leq C_{f,h_H} n^{2(d-1)}\lambda_1(f)^n h_H^+(P).
\]

\item If $\lambda_1(f)=1$, then
\begin{equation}
h_H^+(f^n(P))\leq C_{f,h_H} n^{2d} h_H^+(P).
\end{equation}
\end{enumerate}
\end{theorem}

The zero-entropy exponent $2d$ already occurs in the first author's joint work with Ingram on self-morphisms of abelian varieties \cite{HI}, via the quadratic nature of the N\'eron--Tate height $h_{\NT}$.
Indeed, we showed for abelian varieties that
\begin{equation}\label{eq:HI-ab-var-case}
  h_{\NT}(f^n(P))=
  \begin{cases}
    O\bigl(n^{2(d-1)}\lambda_1(f)^n\bigr),&\lambda_1(f)>1,\\
    O\bigl(n^{2d}\bigr),&\lambda_1(f)=1.
  \end{cases}
\end{equation}
So this fits with \cref{thm:ht-growth-upper-bound} perfectly in the case of abelian varieties.
See \cref{rmk:jump-explanation} for a conceptual explanation on the discrepancy of polynomial exponents.

The proof in \cref{setting:number-field} relies on arithmetic intersection theory and accounts for the vast majority of the paper.
The relative dynamical degree formula is particularly useful in studying various dynamical problems via fibrations \cite{DN11,DNT12,Truong20}.
Inspired by that, Dang, Ghioca, Lesieutre, Satriano, and the first author \cite{DGHLS22} initiated the study of height growth from points to higher-dimensional subvarieties.
Using arithmetic intersection theory, they introduced the notion of higher arithmetic degrees to control the height growth of subvarieties under iterations.
The second author \cite{Song26} further developed this theory and proved a basic result that the limits defining the higher arithmetic dynamical degrees $\widehat{\lambda}_k(f)$ exist for all $k$.
He also established the conjectural formula relating the arithmetic dynamical degrees $\widehat{\lambda}_k(f)$ to the ordinary dynamical degrees $\lambda_k(f)$; see \cite[Theorem~1.5]{Song26}.
Recall that in Dang's proof \cite{Dang20} of the existence of dynamical degrees, he established a higher-codimensional Siu's inequality to get quasi-submultiplicativity of degree sequences.
Similarly, in the arithmetic setting, Yuan \cite{Yuan08} established an arithmetic Siu's inequality and the second author further developed it to a higher-codimensional generalization.

Using Yuan--Zhang's arithmetic Hodge index theorem \cite[Theorem~3.2]{YZ17} (see also \cref{thm:YZ-arithmetic-HIT}), we give an arithmetic realization of an ordered Hodge--Lorentz package and prove in \cref{thm:arithmetic-temporal-estimate} a uniform temporal estimate on arithmetic degrees.
By the same argument in the proof of \cite[Theorem~1.1]{Hu-degree-growth}, we then obtain the following uniform growth estimate of arithmetic degrees in \cref{thm:arithmetic-degree-growth-upper-bound}, as an arithmetic analog of \cite[Theorem~1.1]{Hu-degree-growth}.

Assume that \cref{setting:number-field} holds.
Let $\overline {H}$ be a nef and big adelic $\bQ$-line bundle on $X$ extending $H$. For each $n\geq1$, let $\Gamma_{n}$ be the closure of the graph of $f^n$ in $X\times_K X$ and let $\widetilde{\Gamma}_n$ be its normalization. The two projections induce surjective morphisms $p_n,q_n\colon\widetilde{\Gamma}_n\to X$
satisfying $f^n\circ p_n=q_n$ as rational maps. For $0\leq k\leq d+1$, define the \textit{$k$-th arithmetic degree} of $f^n$ with respect to $\overline{H}$ as
\[
\widehat{\deg}_{k,\overline{H}}(f^n)\coloneqq\big((p_n^*\overline{H})^{d+1-k}\cdot (q_n^*\overline{H})^k\big)_{\widetilde{\Gamma}_{n}}.
\]
The \textit{$k$-th arithmetic dynamical degree} of $f$ is defined as
\[
\widehat{\lambda}_k(f)\coloneqq\lim_{n\to \infty}\widehat{\deg}_{k,\overline{H}}(f^n)^{1/n}.
\]
Note that the limit defining the arithmetic dynamical degree $\widehat{\lambda}_k(f)$ exists and is independent of the choice of nef and big adelic line bundle $\overline{H}$; see \cref{cor:existence-lambda-hat} (cf.~\cite[Corollary~2.8]{Song26}).

By Yuan--Zhang's arithmetic Hodge index theorem \cite[Theorem~3.2]{YZ17}, Ikoma deduced an arithmetic Khovanskii--Teissier inequality \cite[Theorem~2.9(1)]{Ikoma15b}.
In particular, for any $1\leq k\leq d$, we have
\begin{equation}\label{eq:arith-deg-log-concave}
\widehat{\deg}_{k,\overline{H}}(f^n)^2 \geq \widehat{\deg}_{k-1,\overline{H}}(f^n)\widehat{\deg}_{k+1,\overline{H}}(f^n).
\end{equation}
As a consequence, the arithmetic dynamical degree sequence $(\widehat{\lambda}_i(f))_i$ is log-concave, i.e.,
\begin{equation}\label{eq:arith-dyn-deg-log-concave}
\widehat{\lambda}_k(f)^2 \geq \widehat{\lambda}_{k-1}(f)\widehat{\lambda}_{k+1}(f), \qquad 1\leq k\leq d.
\end{equation}

\begin{theorem}\label{thm:arithmetic-degree-growth-upper-bound}
With notation as above, for every maximal affine interval $[r,s]$ of $(\log \widehat{\lambda}_i(f))_{i=0}^{d+1}$, there is a constant $C_{f,\overline{H},r,s}>0$ such that for every $r\leq k\leq s$ and $n\geq1$,
\begin{equation}\label{eq:arithmetic-degree-growth-bound}
\frac{(\overline{H}^{d+1})}{2^{k(d+1-k)}}\widehat{\lambda}_k(f)^n \leq \widehat{\deg}_{k,\overline{H}}(f^n) \leq C_{f,\overline{H},r,s}\, n^{2(k-r)(s-k)}\, \widehat{\lambda}_k(f)^n.
\end{equation}
\end{theorem}


To get the general upper bound in \eqref{eq:arithmetic-degree-growth-bound} over an arbitrary maximal affine interval $[r,s]$, we will first establish it at the endpoints $k=r,s$.
Moreover, at a \textit{strict break} index $k$ of $(\log \widehat{\lambda}_i(f))_{i=0}^{d+1}$, i.e.,
\begin{equation}\label{eq:strict-break-index}
\widehat{\lambda}_k(f) > \widehat{s}_k(f)\coloneqq\sqrt{\widehat{\lambda}_{k-1}(f)\widehat{\lambda}_{k+1}(f)},
\end{equation}
we prove an exact asymptotic formula for $\widehat{\deg}_{k,\overline{H}}(f^n)$.
It is an arithmetic analog of similar formulas for geometric degrees $\deg_{i,H}(f^n)$ at strict breaks; see \cite[Main Theorem]{BFJ08}, \cite[Theorem~1]{DF21}, and \cite[Theorem~1.5]{Hu-degree-growth}.

\begin{theorem}\label{thm:strict-break-asymptotic}
With notation as above, let $k$ be a strict break index of $(\log \widehat{\lambda}_i(f))_{i=0}^{d+1}$.
Then there is a constant $C_{f,\overline{H},k}>0$ such that for every $\widehat{s}_k(f)<\eta<\widehat{\lambda}_k(f)$,
\begin{equation}\label{eq:strict-break-asymptotic}
\widehat{\deg}_{k,\overline{H}}(f^n) = C_{f,\overline{H},k}\, \widehat{\lambda}_k(f)^n + O(\eta^n).
\end{equation}
\end{theorem}

In order to convert the upper bound in \eqref{eq:arithmetic-degree-growth-bound} for $\widehat{\deg}_{1,\overline{H}}(f^n)$ into a genuine upper bound for height growth in \cref{thm:ht-growth-upper-bound}, it remains to prove a uniform height estimate of the form
\[
h_H^+(f^n(P)) \lesssim \widehat{\deg}_{1,\Hbar}(f^n) h_H^+(P).
\]
This will be proved in \cref{thm:ht-degree-bound-N}.
Its function-field counterpart is \cref{thm:ht-growth-bound-F}.

\subsection*{Notations and Terminology}
By a \emph{variety}, we mean a separated, geometrically integral scheme of finite type over a base field.

For a number field $K$, denote by $M_{K,f}$ and $M_{K,\infty}$ the sets of finite and infinite places of $K$, respectively, and set $M_K\coloneqq M_{K,f}\cup M_{K,\infty}$.

For each place $p$ of $\bQ$ in $M_{\bQ}$, let $\overline{\bQ_p}$ denote an algebraic closure of $\bQ_p$; in particular, $\overline{\bQ_\infty}=\bC$.
For finite $p$, denote by $|\cdot|_p$ the unique extension to $\overline{\bQ_p}$ of the $p$-adic absolute value normalized by $|p|_p=1/p$; for $p=\infty$, let $|\cdot|_\infty$ be the usual complex absolute value.
For a place $v$ of $K$ above $p\in M_{\bQ}$, define
\[
|x|_v\coloneqq|\sigma(x)|_p^{[K_v:\bQ_p]},
\qquad x\in K_v,
\]
where $\sigma\colon K_v\hookrightarrow\overline{\bQ_p}$ is any $\bQ_p$-embedding.
We also write $\bQ_v$ for $\bQ_p$ when $v$ lies above $p$.
Note that at a complex place $v$, $|\cdot|_v$ is not actually an absolute value, but $|\cdot|_v^{1/2}$ is.

\subsection*{Acknowledgments}
The authors thank Professors Yohsuke Matsuzawa and Junyi Xie for their interest in this work and comments on the draft.
The first author gratefully acknowledges the hospitality of the Fields Institute during Summer 2026, when part of this work was carried out.


\section{Adelic line bundles and arithmetic dynamical degrees}\label{sec:arithmetic-prelim}

\subsection{Adelic line bundles}
We refer to \cite[\S A.5]{YZ26} for a review of the theory of adelic line bundles on projective varieties, introduced by Zhang in \cite{Zhang95-JAG}.

Let $K$ be a number field, and $X$ be a projective variety of dimension $d$ over $K$. Write $M_K$ for the set of places of $K$ and $X_v^{\mathrm{an}}$ for the analytification of $X_{K_v}$.
An \textit{adelic line bundle} on $X$ is a pair
$\overline{L}=(L,(\|\cdot\|_v)_{v\in M_K})$, where $L$ is a line bundle on $X$ and each $\|\cdot\|_v$ is a continuous metric on $L_{K_v}$ over $X_v^{\mathrm{an}}$ satisfying certain coherence condition; see \cite[\S A.5.2]{YZ26}.
We denote by $\widehat{\mathrm{Pic}}(X)$ the group of isometry classes of adelic line bundles on $X$ and set
\[
\widehat{\mathrm{Pic}}(X)_{\bQ}
\coloneqq
\widehat{\mathrm{Pic}}(X)\otimes_{\bZ}\bQ.
\]
Write $\widehat{\mathrm{Pic}}(X)_{\mathrm{nef}}$ and $\widehat{\mathrm{Pic}}(X)_{\mathrm{int}}$ for the sets of nef and integrable adelic line bundles, respectively.

Let $\overline{L}$ be an adelic line bundle on $X$. For a section
$s\in H^0(X,L)$ and a place $v\in M_K$, define
\[
\|s\|_{v,\,\sup}
\coloneqq
\sup_{x\in X_v^{\mathrm{an}}}\|s(x)\|_v.
\]
The set of \emph{small sections} is
\[
\widehat{H}^0(X,\overline{L})
\coloneqq
\left\{
s\in H^0(X,L):
\|s\|_{v,\,\sup}\leq 1,
\text{ for all } v\in M_K
\right\},
\]
and we denote $\widehat{h}^0(X,\overline{L}) \coloneqq\log\#\widehat{H}^0(X,\overline{L})$.
The \emph{volume} of $\overline{L}$ is defined by
\[
\widehat{\mathrm{vol}}(\overline{L})
\coloneqq
\lim_{m\to\infty}
\frac{(d+1)!}{m^{d+1}}\,
\widehat{h}^0(X,m\overline{L}).
\]
This definition extends to adelic $\bQ$-line bundles by homogeneity; see \cite[\S A.5.4 and \S 5.2]{YZ26}.

An adelic $\bQ$-line bundle $\overline{L}$ on $X$ is called
\emph{big}, if $\widehat{\mathrm{vol}}(\overline{L})>0$.
It is called \emph{pseudo-effective}, if $\overline{L}+\overline{M}$ is big for every big adelic $\bQ$-line bundle $\overline{M}$ on $X$.

Let $\overline{L}$ be an integrable adelic $\bQ$-line bundle on $X$.
For an algebraic point $P\in X(\overline{K})$, viewed as a morphism $\sigma_P\colon\operatorname{Spec}K(P)\to X$, define
\begin{equation}\label{eq:adelic-Weil-ht}
h_{\overline{L}}(P)
\coloneqq
\frac{\widehat{\deg}(\sigma_P^*\overline{L})}{[K(P):\bQ]},
\end{equation}
where $\sigma_P^*\overline{L}\in\widehat{\mathrm{Pic}}(K(P))_{\bQ}$ and
$\widehat{\deg}\colon\widehat{\mathrm{Pic}}(K(P))_{\bQ}\to\bR$
is the arithmetic degree map. The function $h_{\overline{L}}\colon X(\overline{K})\to\bR$ is a Weil height associated with $L$; see \cite[\S A.5.3]{YZ26}.

We will repeatedly use the following standard fact.
See \cite[\S 2.5.5, Lemma~5.2.13, Proposition~4.1.2, and Theorem~5.2.2(1)]{YZ26}.

\begin{lemma}\label{lemma:arithmetic-pullback-positivity}
Let $\varphi\colon Y\to Z$ be a surjective morphism of normal projective varieties.
Then pullback preserves pseudo-effective and nef adelic $\bQ$-line bundles.
If, moreover, $\varphi$ is generically finite, then it also preserves nef and big adelic $\bQ$-line bundles.
\end{lemma}

The following arithmetic Hodge index theorem due to Yuan--Zhang \cite{YZ17} plays a fundamental role in our study of mixed arithmetic intersections and of arithmetic dynamical degrees.
Their proof is by induction on dimension, where the base case is the classical Hodge index theorem for arithmetic surfaces due to \cite{Faltings84,Hriljac85}.
Note that Moriwaki \cite{Moriwaki96} also proved an arithmetic Hodge index theorem in terms of Hermitian line bundles.

\begin{theorem}[{cf.~\cite[Theorem~3.2]{YZ17}}]\label{thm:YZ-arithmetic-HIT}
Let $Y$ be a normal projective variety of dimension $d$ over a number field $K$, and let $\Hbar_1,\dots,\Hbar_{d-1}$ be nef adelic $\bQ$-line bundles on $Y$ such that each $H_i$ is big.
Then for any integrable adelic $\bQ$-line bundle $\Mbar$ on $Y$ satisfying
\[
M\cdot H_1\cdots H_{d-1}=0,
\]
we have
\[
\bigl(\Mbar^2\cdot \Hbar_1\cdots \Hbar_{d-1}\bigr)_Y \leq 0.
\]
\end{theorem}

\begin{corollary}\label{cor:positive_index}
Let $Y$ be a normal projective variety of dimension $d$ over a number field $K$. Let $\overline{H}_1,\ldots,\overline{H}_{d-1}$ be nef adelic $\bQ$-line bundles on $Y$. Let $V_\bQ\subset \widehat{\mathrm{Pic}}(Y)_{\mathrm{int},\bQ}$ be a finite-dimensional $\bQ$-vector space.
Denote $V\coloneqq V_\bQ \otimes_\bQ \bR$.
Then the symmetric bilinear form
\[
\mathscr{B}\colon V\times V\to \bR,
\qquad
(\overline{M}_1,\overline{M}_2)\longmapsto
\bigl(\overline{M}_1\cdot \overline{M}_2\cdot \overline{H}_1\cdots \overline{H}_{d-1}\bigr)_Y
\]
has positive index at most one.
\end{corollary}


\begin{proof}
Let $\overline{N}$ be a nef adelic line bundle on $Y$ with $N$ big. For each $\epsilon\in\bQ_{>0}$, consider the linear map
\[
\theta_\epsilon\colon V\longrightarrow \bR,
\qquad
\overline{M}\longmapsto
M\cdot (H_1+\epsilon N)\cdots (H_{d-1}+\epsilon N).
\]
The subspace $W_\epsilon\coloneqq\ker\theta_\epsilon$ has codimension at most one in $V$. Define the symmetric bilinear form
\[
\mathscr{B}_\epsilon(\overline{M}_1,\overline{M}_2)
\coloneqq
\bigl(\overline{M}_1\cdot\overline{M}_2
\cdot(\Hbar_1+\epsilon\overline{N})\cdots
(\Hbar_{d-1}+\epsilon\overline{N})\bigr)_Y.
\]
Since each $H_i+\epsilon N$ is big, we obtain
$\mathscr{B}_\epsilon(\overline{M},\overline{M})\leq 0$
for every $\overline{M}\in W_\epsilon\cap V_\bQ$ by \cref{thm:YZ-arithmetic-HIT}.
Note that $W_\epsilon\cap V_\bQ$ is dense in $W_\epsilon$.
So, $\mathscr{B}_\epsilon(\overline{M},\overline{M})\leq 0$ for every $\overline{M}\in W_\epsilon$.
Thus, $\mathscr{B}_\epsilon$ has positive index at most one. As $\epsilon\to 0^+$, the forms $\mathscr{B}_\epsilon$ converge to $\mathscr{B}$, so $\mathscr{B}$ also has positive index at most one.
\end{proof}

\subsection{An arithmetic reverse Khovanskii--Teissier inequality for adelic line bundles}

We begin with an elementary bilinear reverse Khovanskii--Teissier inequality.
See \cite{HX24} for a systematic study of rKT property for Lorentzian polynomials.

\begin{lemma}\label{lemma:three-variable-rKT}
Let $V$ be a finite-dimensional $\bR$-vector space and $\mathscr{B}:V\times V\to \bR$ be a symmetric bilinear form with positive index at most one. Then for every $v_1,v_2,v_3\in V$ with $\mathscr{B}(v_i,v_j)\geq 0$ for $1\leq i,j\leq 3$, we have
\[  \sB(v_1,v_1)\sB(v_2,v_3) \leq 2 \sB(v_1, v_2)\sB(v_1, v_3).\]
\end{lemma}
\begin{proof}
If $\sB(v_1,v_1)=0$, the result follows immediately. We may therefore assume that $\sB(v_1,v_1)>0$. Let $v_1^\perp\subset V$ denote the orthogonal complement of $v_1$ with respect to $\sB$. Since $\sB$ has positive index at most one, its restriction to $v_1^\perp$ is negative semidefinite. For $i=2,3$, take $u_i\coloneqq v_i-c_iv_1\in v_1^\perp$, where $c_i=\sB(v_i,v_1)/\sB(v_1,v_1)\geq0$. Then
\[
0\leq -\sB(u_i,u_i) = c_i^2\sB(v_1,v_1)-\sB(v_i,v_i)\leq c_i^2\sB(v_1,v_1).
\]
By the Cauchy--Schwarz inequality, we have
\[
\begin{aligned}
\sB(v_2,v_3) &= \sB(u_2,u_3)+c_2c_3\sB(v_1,v_1)\\
&\leq \sqrt{-\sB(u_2,u_2)}\sqrt{-\sB(u_3,u_3)}+c_2c_3\sB(v_1,v_1)\\
&\leq 2c_2c_3\sB(v_1,v_1).
\end{aligned}
\]
Multiplying both sides by $\sB(v_1,v_1)$ completes the proof.
\end{proof}

We derive the following arithmetic reverse Khovanskii--Teissier inequality for adelic line bundles with a non-optimal coefficient.

\begin{theorem}[(Non-optimal) arithmetic reverse Khovanskii--Teissier inequality]\label{thm:rKT}
Let $Y$ be a normal projective variety of dimension $d$ over a number field $K$ and $\overline{L},\overline{N}_1,\dots,\overline{N}_{d+1}$ be nef adelic $\bQ$-line bundles. Then for every $0\leq k\leq d+1$, 
\begin{equation}\label{eq:arithmetic-rKT}
\bigl(\overline{L}^{d+1}\bigr)\bigl(\overline{N}_1\cdot\overline{N}_2\cdots\overline{N}_{d+1}\bigr)\leq 2^{k(d+1-k)}\bigl(\overline{N}_1\cdots\overline{N}_k\cdot \overline{L}^{d+1-k}\bigr)\bigl(\overline{L}^{k}\cdot\overline{N}_{k+1}\cdots\overline{N}_{d+1}\bigr).
\end{equation}
\end{theorem}
\begin{proof}
For $k=0$ or $d+1$, the inequality is an equality. We may assume and fix $1\leq k\leq d$.
Let $\overline{H}$ be a nef and big adelic line bundle on $Y$.
Fix $\epsilon\in\bQ_{>0}$ and set $\overline{L}_\epsilon \coloneqq\overline{L}+\epsilon\overline{H}$, $\overline{N}_{i,\epsilon} \coloneqq\overline{N}_i+\epsilon\overline{H}$ for $1\leq i\leq d+1$.
For $0\leq i\leq k$ and $0\leq j\leq d+1-k$, define
\[
a_{i,j}\coloneqq
\bigl(\overline{N}_{1,\epsilon}\cdots\overline{N}_{i,\epsilon}
\cdot\overline{N}_{k+1,\epsilon}\cdots\overline{N}_{k+j,\epsilon}
\cdot\overline{L}_\epsilon^{d+1-i-j}\bigr).
\] 
By the choice of $\overline{H}$, we have $a_{i,j}\geq\epsilon^{d+1}(\overline{H}^{d+1})>0$ for all $i,j$. 

We claim that $a_{i,j}a_{i+1,j+1}\leq 2a_{i+1,j}a_{i,j+1}$ for all $0\leq i\leq k-1$, $0\leq j\leq d-k$. In fact, consider the $\bQ$-vector space $V_\bQ\subset \widehat{\mathrm{Pic}}(Y)_{\mathrm{int},\bQ}$ spanned by $\overline{L}_{\epsilon},\overline{N}_{1,\epsilon},\dots,\overline{N}_{d+1,\epsilon}$ and the symmetric bilinear form $\sB_{i,j}$ on $V\coloneqq V_\bQ \otimes_\bQ \bR$ defined by
\[
\sB_{i,j}(\overline{M}_1,\overline{M}_2) \coloneqq \bigl(\overline{M}_1\cdot\overline{M}_2\cdot\overline{N}_{1,\epsilon}\cdots\overline{N}_{i,\epsilon}\cdot \overline{N}_{k+1,\epsilon}\cdots\overline{N}_{k+j,\epsilon}\cdot\overline{L}_{\epsilon}^{d-i-j-1}\bigr).
\]
By \cref{cor:positive_index}, $\sB_{i,j}$ has positive index at most one. So, applying \cref{lemma:three-variable-rKT} to $\overline{L}_\epsilon$, $\overline{N}_{i+1,\epsilon}$, and $\overline{N}_{k+j+1,\epsilon}$ proves the claim.

Multiplying these $k(d+1-k)$ inequalities and canceling the common terms, we obtain
\[\frac{a_{0,0}a_{k,d+1-k}}{a_{k,0}a_{0,d+1-k}}=\prod_{0\leq i\leq k-1}\,\prod_{0\leq j\leq d-k}\frac{a_{i,j}a_{i+1,j+1}}{a_{i+1,j}a_{i,j+1}}\leq 2^{k(d+1-k)}.\]
It follows that
\[
\bigl(\overline{L}_{\epsilon}^{d+1}\bigr)\bigl(\overline{N}_{1,\epsilon}\cdots\overline{N}_{d+1,\epsilon}\bigr)\leq 2^{k(d+1-k)}\bigl(\overline{N}_{1,\epsilon}\cdots\overline{N}_{k,\epsilon}\cdot \overline{L}_{\epsilon}^{d+1-k}\bigr)\bigl(\overline{L}_{\epsilon}^{k}\cdot\overline{N}_{k+1,\epsilon}\cdots\overline{N}_{d+1,\epsilon}\bigr).
\]
Letting $\epsilon\to0^+$, we prove the desired inequality \eqref{eq:arithmetic-rKT}.
\end{proof}

\begin{remark}
It was Xiao \cite[Remark~3.1]{Xiao15} first observed such class of inequalities in K\"ahler geometry. For the terminology, see \cite[Remark~4.18]{LX16}.
The sharp three-class inequality has appeared in \cite[Proposition~1.15]{BFJ08} for surfaces, and is proved in \cite[Theorem~5.7]{LX17} for arbitrary compact K\"ahler manifold.
Dang \cite[Theorem~3.4.3]{Dang20} established its algebraic counterpart for nef divisors with a non-optimal coefficient; using this he gave a new proof of the existence of dynamical degrees in arbitrary characteristic.
In the algebraic situation, the sharp fully mixed inequality was proved by Jiang--Li \cite[Theorem~3.5]{JL23} using multipoint Okounkov bodies.
Although the coefficient $2^{k(d+1-k)}$ in \eqref{eq:arithmetic-rKT} is far from being optimal, the point is that it is uniform which suffices to prove the existence of arithmetic dynamical degrees.
\end{remark}

\subsection{Arithmetic dynamical degrees}
Let $X$ be a normal projective variety of dimension $d$ over a number field $K$ and $f\colon X\dashrightarrow X$ be a dominant rational self-map. For each $n\geq1$, let $\Gamma_{n}$ be the graph of $f^n$ in $X\times_K X$ and let $\widetilde{\Gamma}_n$ be its normalization. The two projections induce surjective morphisms $p_n,q_n\colon\widetilde{\Gamma}_n\to X$
satisfying $f^n\circ p_n=q_n$ as rational maps.

\begin{definition}
Fix a nef and big adelic $\bQ$-line bundle $\overline{H}$ on $X$.
For $0\leq k\leq d+1$ and $n\geq1$, define the \textit{$k$-th arithmetic degree} of $f^n$ with respect to $\overline{H}$ as
\[
\widehat{\deg}_{k,\overline{H}}(f^n)\coloneqq\big((p_n^*\overline{H})^{d+1-k}\cdot (q_n^*\overline{H})^k\big)_{\widetilde{\Gamma}_{n}}.
\]
By convention, set $\widehat{\deg}_{k,\overline{H}}(f^0)\coloneqq (\Hbar^{d+1})$.
The \textit{$k$-th arithmetic dynamical degree} of $f$ with respect to $\overline{H}$ is defined as
\[\widehat{\lambda}_k(f)\coloneqq\lim_{n\to \infty}\widehat{\deg}_{k,\overline{H}}(f^n)^{1/n}.\]
\end{definition}

\begin{remark}
Let $K'/K$ be a finite extension, and let $\overline{H}_1,\dots,\overline{H}_{d+1}$ be integrable adelic $\bQ$-line bundles on $X$. Set $X_{K'}\coloneqq X\times_K K'$,
and let $\overline{H}_{i,K'}$ denote the pullback of $\overline{H}_i$ to $X_{K'}$. By the projection formula \cite[Proposition~4.1.2]{YZ26},
\[
\bigl(\overline{H}_{1,K'}\cdots\overline{H}_{d+1,K'}\bigr)_{X_{K'}}= [K':K]\bigl(\overline{H}_1\cdots\overline{H}_{d+1}\bigr)_X.
\]
In particular, for the base change $f_{K'}$ of $f$, we have
\[
\widehat{\deg}_{k,\overline{H}_{K'}}(f_{K'})= [K':K]\,\widehat{\deg}_{k,\overline{H}}(f).
\]
In contrast, geometric intersection numbers are invariant under extension of the base field, so $\deg_{k,H_{K'}}(f_{K'})=\deg_{k,H}(f)$.
\end{remark}

When $\overline{H}$ is induced by a Hermitian line bundle on an arithmetic model $\cX$ of $X$ over $\cO_K$, the limit defining $\widehat{\lambda}_k(f)$ exists by \cite[Corollary~2.8]{Song26}.
For completeness, we show below that this limit also exists for any nef and big adelic line bundle $\overline{H}$ and does not depend on the choice of $\overline{H}$. Here, we use the newly developed (non-optimal) arithmetic reverse Khovanskii--Teissier inequality to prove the existence of $\widehat{\lambda}_k(f)$.
This approach differs from that in \cite{Song26}.
Indeed, \cite[Theorem~2.5]{Song26} relies on the so-called arithmetic Demailly approximation theorem, but yields a stronger coefficient than our \cref{thm:rKT}.

\begin{theorem}[Quasi-submultiplicativity of arithmetic degrees]\label{thm:quasi-submultiplicity}
Let $X$ be a normal projective variety of dimension $d$ over a number field $K$ and $f,g\colon X\dashrightarrow X$ be dominant rational self-maps. Let $\overline{H}$ be a nef and big adelic $\bQ$-line bundle. Then for every $0\leq k\leq d+1$, 
\[
\widehat{\deg}_{k,\overline{H}}(f\circ g)\leq \frac{2^{k(d+1-k)}}{\bigl(\overline{H}^{d+1}\bigr)}\,\widehat{\deg}_{k,\overline{H}}(f)\,\widehat{\deg}_{k,\overline{H}}(g).
\]
\end{theorem}
\begin{proof}
The case $k=0$ follows directly from the definition of the arithmetic degree, so fix $1\leq k\leq d+1$.
Consider the graph closure in $X^3$ of the following rational map
\[
(g,f\circ g)\colon X\ratmap X^2, \qquad x\longmapsto (g(x),f(g(x))).
\]
Let $Y$ be its normalization with induced surjective morphisms $\pi_0,\pi_1,\pi_2\colon Y\to X$ satisfying that $\pi_1=g\circ\pi_0$ and $\pi_2=f\circ\pi_1$.
Denote $\Hbar_i\coloneqq \pi_i^*\Hbar$ for $i=0,1,2$, which are still nef and big.
Then the projection formula \cite[Proposition~4.1.2]{YZ26} yields that
\begin{equation}\label{eq:intersection-degree-identities}
\begin{aligned}
\bigl(\Hbar_1^{d+1}\bigr)_Y &= \deg(g)\bigl(\Hbar^{d+1}\bigr), \qquad & \bigl(\Hbar_0^{d+1-k}\cdot\Hbar_2^{k}\bigr)_Y &= \widehat{\deg}_{k,\Hbar}(f\circ g),\\
\bigl(\Hbar_1^{d+1-k}\cdot\Hbar_2^{k}\bigr)_Y &= \deg(g)\,\widehat{\deg}_{k,\Hbar}(f),\qquad &\bigl(\Hbar_0^{d+1-k}\cdot\Hbar_1^{k}\bigr)_Y &= \widehat{\deg}_{k,\Hbar}(g).
\end{aligned}
\end{equation}
By applying \cref{thm:rKT} with $\Lbar=\Hbar_1$, $\Nbar_1=\dots=\Nbar_k=\Hbar_2$, and $\Nbar_{k+1}=\dots=\Nbar_{d+1}=\Hbar_0$, we get
\[
\bigl(\Hbar_1^{d+1}\bigr)_Y\bigl(\Hbar_0^{d+1-k}\cdot\Hbar_2^{k}\bigr)_Y
\leq 2^{k(d+1-k)} \bigl(\Hbar_1^{d+1-k}\cdot\Hbar_2^k\bigr)_Y\bigl(\Hbar_0^{d+1-k}\cdot\Hbar_1^{k}\bigr)_Y.
\]
By substituting the identities \eqref{eq:intersection-degree-identities} into the above inequality, we obtain
\begin{align*}
\widehat{\deg}_{k,\overline{H}}(f\circ g) &\leq \frac{2^{k(d+1-k)}}{\bigl(\overline{H}_1^{d+1}\bigr)_Y}\,\deg(g)\,\widehat{\deg}_{k,\overline{H}}(f)\,\widehat{\deg}_{k,\overline{H}}(g)\\
&= \frac{2^{k(d+1-k)}}{\bigl(\overline{H}^{d+1}\bigr)}\,\widehat{\deg}_{k,\overline{H}}(f)\,\widehat{\deg}_{k,\overline{H}}(g).\qedhere
\end{align*}
\end{proof}

\begin{corollary}\label{cor:existence-lambda-hat}
Let $X$ be a normal projective variety of dimension $d$ over a number field $K$, and let $f\colon X\dashrightarrow X$ be a dominant rational self-map. Then, for $1\leq k\leq d+1$, the limit
\[
\widehat{\lambda}_k(f)\coloneqq \lim_{n\to\infty}\widehat{\deg}_{k,\overline{H}}(f^n)^{1/n} = \inf_{n\geq1} \left(\frac{2^{k(d+1-k)}}{\bigl(\overline{H}^{d+1}\bigr)}\,\widehat{\deg}_{k,\overline{H}}(f^n)\right)^{1/n}
\]
exists and is independent of the choice of the nef and big adelic $\bQ$-line bundle $\overline{H}$ on $X$.
\end{corollary}
\begin{proof}
By \cref{thm:quasi-submultiplicity}, for all $n,m\geq 1$, we have
\[
\widehat{\deg}_{k,\overline{H}}(f^{n+m})\leq \frac{2^{k(d+1-k)}}{\bigl(\overline{H}^{d+1}\bigr)}\,\widehat{\deg}_{k,\overline{H}}(f^n)\,\widehat{\deg}_{k,\overline{H}}(f^m).
\]
Thus, the sequence $\Bigl(\frac{2^{k(d+1-k)}}{(\overline{H}^{d+1})}\,\widehat{\deg}_{k,\overline{H}}(f^n)\Bigr)_{n\geq1}$ is submultiplicative. The existence of the limit then follows from Fekete's lemma.

Let $\overline{L}$ be another nef and big adelic $\bQ$-line bundle on $X$.
Let $\Gamma\subset X\times_K X$ be the closure of the graph of $f$, and let $\widetilde{\Gamma}$ be its normalization. Denote the two projections by
$\pi_1,\pi_2\colon\widetilde{\Gamma}\to X$. Applying \cref{thm:rKT} twice and using the projection formula \cite[Proposition~4.1.2]{YZ26}, we obtain
\begin{align*}
\deg(\pi_2)\bigl(\Lbar^{d+1}\bigr)^2\,\widehat{\deg}_{k,\Hbar}(f) &= \bigl(\pi_2^*\Lbar^{d+1}\bigr)\bigl(\pi_1^*\Lbar^{d+1}\bigr)\,\widehat{\deg}_{k,\Hbar}(f)\\
&= \bigl(\pi_2^*\Lbar^{d+1}\bigr)\bigl(\pi_1^*\Lbar^{d+1}\bigr)\big(\pi_1^*\Hbar^{d+1-k}\cdot\pi_2^*\Hbar^k\big)\\
&\leq 2^{k(d+1-k)} \bigl(\pi_2^*\Lbar^{d+1}\bigr)\big(\pi_1^*\Lbar^{d+1-k}\cdot\pi_2^*\Hbar^{k}\big)\bigl(\pi_1^*\Lbar^{k}\cdot\pi_1^*\Hbar^{d+1-k}\bigr)\\
&= 2^{k(d+1-k)} \bigl(\Lbar^{k}\cdot\Hbar^{d+1-k}\bigr)\bigl(\pi_2^*\Lbar^{d+1}\bigr)\big(\pi_1^*\Lbar^{d+1-k}\cdot\pi_2^*\Hbar^{k}\big)\\
&\leq 2^{2k(d+1-k)} \bigl(\Lbar^{k}\cdot\Hbar^{d+1-k}\bigr)\bigl(\pi_2^*\Lbar^{d+1-k}\cdot\pi_2^*\Hbar^{k}\bigr)\big(\pi_1^*\Lbar^{d+1-k}\cdot\pi_2^*\Lbar^{k}\big)\\
&= 2^{2k(d+1-k)} \deg(\pi_2) \bigl(\Lbar^{k}\cdot\Hbar^{d+1-k}\bigr)\bigl(\Lbar^{d+1-k}\cdot\Hbar^{k}\bigr)\,\widehat{\deg}_{k,\Lbar}(f),
\end{align*}
equivalently,
\[
\widehat{\deg}_{k,\Hbar}(f)\leq \frac{2^{2k(d+1-k)}\big(\Lbar^k\cdot \Hbar^{d+1-k}\big)\big(\Lbar^{d+1-k}\cdot\Hbar^{k}\big)}{\big(\Lbar^{d+1}\big)^2}\,\widehat{\deg}_{k,\Lbar}(f).
\]
The constant in this inequality is independent of $f$. Replacing $f$ by $f^n$, taking $n$-th roots, and letting $n\to\infty$, we prove
\[
\lim_{n\to\infty}
\widehat{\deg}_{k,\overline{H}}(f^n)^{1/n}
\leq
\lim_{n\to\infty}
\widehat{\deg}_{k,\overline{L}}(f^n)^{1/n}.
\]
Exchanging $\overline{H}$ and $\overline{L}$ proves the reverse inequality, so the two limits are equal.
\end{proof}

Since every Hermitian line bundle defines an adelic line bundle, the definition of arithmetic dynamical degrees using adelic line bundles is the same as that in \cite{Song26}, which uses Hermitian line bundles.
Moreover, by \cite[Theorem 1.5]{Song26}, we have the relative degree formula
\begin{equation}\label{eq:relative-degree-formula}
\widehat{\lambda}_k(f)=\max\{\lambda_k(f),\lambda_{k-1}(f)\}
\end{equation}
for all $1\leq k\leq d$.
As a straightforward consequence, we have the following relation between maximal affine intervals of $(\log\widehat{\lambda}_i(f))_{i=0}^{d+1}$ and of $(\log\lambda_i(f))_{i=0}^d$.

\begin{lemma}\label{lem:maximal_affine_interval-N}
Let $X$ be a normal projective variety of dimension $d$ over a number field $K$, and let $f\colon X\dashrightarrow X$ be a dominant rational self-map.
\begin{enumerate}[label=\textnormal{(\arabic*)}]
\item If $\lambda_1(f)>1$, then $[0,s]$ is a maximal affine interval of $(\log\lambda_i(f))_{i=0}^d$ if and only if it is also a maximal affine interval of $(\log\widehat{\lambda}_i(f))_{i=0}^{d+1}$.

\item If $\lambda_1(f)=1$, then $\widehat{\lambda}_i(f)=1$ for all $0\leq i\leq d+1$; in particular, $[0,d+1]$ is the maximal affine interval of $(\log\widehat{\lambda}_i(f))_{i=0}^{d+1}$.
\end{enumerate}
\end{lemma}
\begin{proof}
By log-concavity of dynamical degrees, there exist unique $0\leq p\leq q\leq d$ such that
\[
1=\lambda_0<\lambda_1<\dots<\lambda_p=\dots=\lambda_q>\dots>\lambda_d=\deg(f).
\]
Similarly, by log-concavity of arithmetic dynamical degrees \eqref{eq:arith-dyn-deg-log-concave}, we also have unique $0\leq p'\leq q'\leq d+1$ such that
\[
1=\widehat{\lambda}_0<\widehat{\lambda}_1<\dots<\widehat{\lambda}_{p'}=\dots=\widehat{\lambda}_{q'}>\dots>\widehat{\lambda}_{d+1}=\deg(f).
\]
Now thanks to the relative degree formula \eqref{eq:relative-degree-formula}, one has $p=p'$ and $q'=q+1$.
\cref{lem:maximal_affine_interval-N} thus follows.
\end{proof}

The following comparison lemma, which will be used in the proof of \cref{thm:ht-degree-bound-N}, has appeared in the proof of \cite[Theorem~1.5]{Song26} for arithmetic degrees defined by Hermitian line bundles. For completeness, we include a proof in the current adelic setting.

\begin{lemma}\label{lemma:degree-comparison}
Let $X$ be a normal projective variety of dimension $d$ over a number field $K$, and let $\overline H$ be a nef and big $\bQ$-adelic line bundle on $X$. 
Then for every dominant rational self-map $g\colon X\dashrightarrow X$ and every $0\leq k\leq d$, we have
\begin{equation}\label{eq:arith-degree-vs-geom-degree}
\widehat{\deg}_{k,\overline H}(g)\geq \frac{\bigl(\Hbar^{d+1}\bigr)}{(d+1)(H^d)}\deg_{k,H}(g).
\end{equation}
\end{lemma}

\begin{proof}
Let $\pi\colon X\to\operatorname{Spec}K$ be the structural morphism. Choose a nef adelic line bundle $\overline E\in\widehat{\mathrm{Pic}}(K)$ with $\widehat{\deg}(\overline E)=1$.
Since $\overline H$ is nef and big, we have $(\overline H^{d+1})>0$. Fix a rational number $t>(d+1)(H^d)/(\overline H^{d+1})$.
Yuan's inequality \cite[Theorem~5.2.2(2)]{YZ26} then implies
\begin{align*}
\widehat{\mathrm{vol}}\bigl(t\overline H-\pi^*\overline E\bigr) &\geq t^{d+1}(\overline H^{d+1}) - (d+1)t^d \bigl(\overline H^d\cdot\pi^*\overline E\bigr)\\
&=t^d\bigl(t(\overline H^{d+1})-(d+1)(H^d)\bigr)>0.
\end{align*}
Thus $t\overline H-\pi^*\overline E$ is big.

Fix a dominant rational self-map $g\colon X\ratmap X$ and $0\leq k\leq d$.
Let $\Gamma\subset X\times_K X$ be the closure of the graph
of $g$, let $\widetilde\Gamma$ be its normalization, and denote
the two induced morphisms by $\pi_1,\pi_2\colon\widetilde\Gamma\to X$. Since both $\pi_i^*\Hbar$ are nef and big, by \cite[Lemma~5.2.13(b)]{YZ26}, we have
\[
\begin{aligned}
\widehat{\deg}_{k,\overline H}(g)&=\bigl((\pi_1^*\overline H)^{d+1-k} \cdot (\pi_2^*\overline H)^k\bigr)_{\widetilde{\Gamma}}
\geq\frac{1}{t}\bigl((\pi_1^*\overline H)^{d-k} \cdot (\pi_2^*\overline H)^k\cdot\pi_1^*\pi^*\overline E\bigr)_{\widetilde{\Gamma}}\\
&=\frac{\widehat{\deg}(\Ebar)}{t}\bigl((\pi_1^*H)^{d-k} \cdot (\pi_2^*H)^k\bigr)_{\widetilde{\Gamma}}
=\frac{1}{t}\deg_{k,H}(g).
\end{aligned}
\]
As $t$ is arbitrary, letting $t$ approximate $(d+1)(H^d)/(\overline H^{d+1})$ proves \eqref{eq:arith-degree-vs-geom-degree}.
\end{proof}

\begin{remark}
Similar to the proof of the (non-optimal) arithmetic reverse Khovanskii--Teissier inequality \cref{thm:rKT}, one can also use \cref{lemma:three-variable-rKT} to prove \cref{lemma:degree-comparison} with the cost of a weaker uniform constant $\frac{(\overline H^{d+1})}{2^d(H^d)}$, still independent of the map $g$.
\end{remark}

\section{Arithmetic realization of ordered Hodge--Lorentz packages}\label{sec:arithmetic-temporal-estimate}

In this section, we realize the arithmetic data on finite-time orbit spaces as ordered Hodge--Lorentz packages.
The related positivity notions are arithmetic nefness and arithmetic pseudo-effectiveness.
Then by applying the abstract Hodge--Lorentz package theory developed in \cite[\S 2]{Hu-degree-growth} which builds on \cite{HJ-degree-growth}, we prove the following uniform temporal estimate \cref{thm:arithmetic-temporal-estimate} on arithmetic degrees.
We then deduce \cref{thm:strict-break-asymptotic,thm:arithmetic-degree-growth-upper-bound} from it.

\begin{theorem}\label{thm:arithmetic-temporal-estimate}
Let $f\colon X\ratmap X$ be a dominant rational self-map of a normal projective variety $X$ of dimension $d$ over a number field $K$, and let $\Hbar$ be a nef and big adelic $\bQ$-line bundle on $X$.
Then for every $1\leq k\leq d$, there exists a constant $C'_{f,\Hbar,k}>0$ such that for all $n\geq1$,
\begin{equation}\label{eq:arithmetic-temporal-estimate}
\left|\log\frac{\widehat{\deg}_{k,\Hbar}(f^n)^2}{\widehat{\deg}_{k,\Hbar}(f^{n-1})\,\widehat{\deg}_{k,\Hbar}(f^{n+1})}\right|
\leq C'_{f,\Hbar,k} \frac{\sqrt{\widehat{\deg}_{k-1,\Hbar}(f^n)\,\widehat{\deg}_{k+1,\Hbar}(f^n)}}{\widehat{\deg}_{k,\Hbar}(f^n)}.
\end{equation}
\end{theorem}

\begin{definition}\label{def:ordered-HL}
Let $V$ be a finite-dimensional real vector space, let $\cN\subseteq\cP\subseteq V$ be nonempty convex cones, and let $Q\colon V^{d+1}\to\bR$ be a symmetric multilinear form.
We call $(V,\cN,\cP,Q)$ an \emph{ordered Hodge--Lorentz package} of degree $d+1$ if the following conditions hold:
\begin{enumerate}[label=\textnormal{(\roman*)}]
\item \label{def:ordered-HL-i} Nonnegativity: For every $\xi\in\cP$ and $\eta_2,\ldots,\eta_{d+1}\in\cN$,
\[
Q(\xi,\eta_2,\ldots,\eta_{d+1})\geq0.
\]
\item \label{def:ordered-HL-ii} Lorentzian: For every $\eta_3,\ldots,\eta_{d+1}\in\cN$, the symmetric bilinear form
\[
\sB\colon V\times V \to \bR, \quad (v_1,v_2)\longmapsto Q(v_1,v_2,\eta_3,\ldots,\eta_{d+1})
\]
has positive index at most one.
\end{enumerate}
We write $v_1\preceq_{\cP}v_2$ if $v_2-v_1\in\cP$; this is the preorder induced by $\cP$.
\end{definition}

\begin{lemma}\label{lemma:arithmetic-realization}
Let $Y$ be a normal projective variety of dimension $d$ over a number field $K$.
Let $V_\bQ \subset \widehat{\mathrm{Pic}}(Y)_{\mathrm{int},\bQ}$ be a finite-dimensional $\bQ$-vector space.
Denote $V\coloneqq V_\bQ \otimes_\bQ \bR$.
Let $\cN_V\subset V$ (resp. $\cP_V\subset V$) be the closure of the cone generated by nef (resp. pseudo-effective) adelic $\bQ$-line bundles containing in $V$.
Then the tuple $(V,\cN_V,\cP_V,Q)$, where by extending scalars,
\[
Q(\Mbar_1,\ldots,\Mbar_{d+1}) \coloneqq \bigl(\Mbar_1\cdots\Mbar_{d+1}\bigr)_Y,
\]
is an ordered Hodge--Lorentz package of degree $d+1$.
\end{lemma}
\begin{proof}
It is well known that for an adelic $\bQ$-line bundle, if it is nef then it is pseudo-effective; see \cite[Lemma~5.2.13(a)]{YZ26}.
Hence we have $\cN_V\subset \cP_V$.
By continuity, it suffices to verify the nonnegativity condition~\ref{def:ordered-HL-i} for adelic $\bQ$-line bundles, which is exactly \cite[Lemma~5.2.13(b)]{YZ26}.
The Lorentzian signature condition~\ref{def:ordered-HL-ii} has been proved in \cref{cor:positive_index} for adelic $\bQ$-line bundles.
Again the general case follows by approximation.
\end{proof}

Recall the following abstract estimate for the discrete logarithmic Laplacian associated with an ordered Hodge--Lorentz package.

\begin{theorem}[{see \cite[Theorem~2.6]{Hu-degree-growth}}]
\label{thm:abstract-log-laplacian}
Fix integers $d\geq1,N\geq2$.
Let $(V,\cN,\cP,Q)$ be an ordered Hodge--Lorentz package of degree $d+1$, and let $\eta_0,\ldots,\eta_N\in\cN$.
Suppose that there exist constants $C\geq1$ and $\lambda>0$, and positive numbers
\[
\delta_i(m)>0,\qquad
0\leq i\leq d+1,\ 0\leq m\leq N,
\]
such that the following hypotheses hold.

\begin{enumerate}[label=\textnormal{(\roman*)}]
\item\label{thm:abstract-log-laplacian-comparison}
For every $0\leq m<N$,
\[
C^{-1}\eta_m\preceq_{\cP}\eta_{m+1}\preceq_{\cP} C\eta_m.
\]

\item\label{thm:abstract-log-laplacian-two-time-shift}
For every $0\leq m\leq\ell\leq N$ and $0\leq i\leq d+1$,
\[
Q\bigl(\eta_m^{[d+1-i]},\eta_\ell^{[i]}\bigr)=\lambda^m\delta_i(\ell-m).
\]
\end{enumerate}
Then for every $1\leq k\leq d$ and $1\leq n\leq N-1$,
\begin{equation}\label{eq:abstract-log-laplacian}
\left|\log\frac{\delta_k(n)^2}{\delta_k(n-1)\delta_k(n+1)}\right|
\leq 2k(d+1-k)C^{2d+6}\frac{\sqrt{\delta_{k-1}(n)\delta_{k+1}(n)}}{\delta_k(n)}.
\end{equation}
\end{theorem}

Suppose that $f\colon X\ratmap X$ is a dominant rational self-map of a normal projective variety $X$ of dimension $d$ over a number field $K$.
For each fixed $N\geq1$, let $\Gamma_{[N]}$ be the $N$-th iterated graph of $f$, i.e., the closure of the graph of the rational map
\[
X\ratmap X^N, \qquad x\mapsto (f(x),\dots,f^N(x))
\]
in $X^{N+1}$.
Let $Y_N$ denote the normalization of $\Gamma_{[N]}$, which is a normal projective variety of dimension $d$ over $K$.
For $0\leq m\leq N$, each coordinate projection from $X^{N+1}$ to $X$ induces a morphism
\[
\pi_{m,N}\colon Y_N \to X.
\]
Note that $\pi_{0,N}$ is birational and $\pi_{m,N}=f^m\circ\pi_{0,N}$ as rational maps.
In particular, each $\pi_{m,N}$ is surjective and generically finite with
\[
\deg(\pi_{m,N})=\deg(f^m)=\lambda_d(f)^m.
\]
For $0\leq m\leq N$, denote
\[
\Hbar_{m,N}\coloneqq\pi_{m,N}^*\Hbar.
\]
Then by \cref{lemma:arithmetic-pullback-positivity}, each $\Hbar_{m,N}$ is a nef and big adelic $\bQ$-line bundle on $Y_N$.

\begin{lemma}\label{lemma:arithmetic-two-time-shift}
With notation as above, for every $0\leq i\leq d+1$ and $0\leq m\leq\ell\leq N$,
\begin{equation}\label{eq:arithmetic-two-time-shift}
\bigl(\Hbar_{m,N}^{d+1-i}\cdot \Hbar_{\ell,N}^i\bigr)_{Y_N} = \lambda_d(f)^m\,\widehat{\deg}_{i,\overline{H}}(f^{\ell-m}).
\end{equation}
\end{lemma}
\begin{proof}
Fix $0\leq i\leq d+1$ and $0\leq m\leq\ell\leq N$.
If $m=\ell$, then the projection formula \cite[Proposition~4.1.2]{YZ26} yields that
\[
\bigl(\Hbar_{m,N}^{d+1}\bigr)_{Y_N}=\deg(\pi_{m,N})\bigl(\Hbar^{d+1}\bigr)=\lambda_d(f)^m\,\widehat{\deg}_{i,\overline{H}}(f^{0}).
\]
So we assume that $m<\ell$ and denote $k=\ell-m$.
Recall from the proof of \cite[Lemma~3.2]{Hu-degree-growth} that the morphism $(\pi_{m,N},\pi_{\ell,N})\colon Y_N\to X\times X$ factors through $\widetilde\Gamma_{k}$, the normalization of the graph closure of $f^{k}$.
Denote it by $\varphi_{m,\ell}^{(N)}\colon Y_N\to \widetilde\Gamma_{k}$.
Then we have $\pi_{m,N}=p_k\circ\varphi_{m,\ell}^{(N)}$ and $\pi_{\ell,N}=q_k\circ\varphi_{m,\ell}^{(N)}$.
It again follows from the projection formula that
\[
\bigl(\Hbar_{m,N}^{d+1-i}\cdot \Hbar_{\ell,N}^i\bigr)_{Y_N} = \deg\bigl(\varphi_{m,\ell}^{(N)}\bigr)\bigl((p_{k}^*\Hbar)^{d+1-i}\cdot (q_{k}^*\Hbar)^{i}\bigr)_{\widetilde\Gamma_{k}} = \lambda_d(f)^m\,\widehat{\deg}_{i,\overline{H}}(f^{k}).\qedhere
\]
\end{proof}

\begin{lemma}\label{lemma:arithmetic-adjacent-comparison}
With notation as above, there is a constant $C_{f,\Hbar}\geq1$ such that for every $N\geq1$ and $0\leq m<N$,
\begin{equation}\label{eq:projective-adjacent-comparison}
C_{f,\Hbar}^{-1}\,\Hbar_{m,N}\preceq_{Y_N}\Hbar_{m+1,N}\preceq_{Y_N}C_{f,\Hbar}\,\Hbar_{m,N}.
\end{equation}
\end{lemma}
\begin{proof}
Choose a positive integer $t$ such that
\[
t>\frac{d+1}{\bigl(\overline{H}^{d+1}\bigr)}
\max\left\{
\widehat{\deg}_{1,\overline{H}}(f),
\frac{\widehat{\deg}_{d,\overline{H}}(f)}{\lambda_d(f)}
\right\}.
\]
By Yuan's inequality \cite[Theorem~5.2.2(2)]{YZ26} and \cref{lemma:arithmetic-two-time-shift}, we have
\begin{align*}
\widehat{\mathrm{vol}}\bigl(t\overline H_{m,N}-\overline H_{m+1,N}\bigr)
&\geq t^{d+1}\bigl(\overline H_{m,N}^{d+1}\bigr)
-(d+1)t^d
\bigl(\overline H_{m,N}^d\cdot\overline H_{m+1,N}\bigr)\\
&=t^d\lambda_d(f)^m\bigl(\overline{H}^{d+1}\bigr)
\left(t - \frac{d+1}{\bigl(\overline{H}^{d+1}\bigr)} \,\widehat{\deg}_{1,\overline H}(f)\right)>0,
\end{align*}
and similarly,
\begin{align*}
\widehat{\mathrm{vol}}\bigl(t\overline H_{m+1,N}-\overline H_{m,N}\bigr)
&\geq t^{d+1}\bigl(\overline H_{m+1,N}^{d+1}\bigr)
-(d+1)t^d
\bigl(\overline H_{m+1,N}^d\cdot\overline H_{m,N}\bigr)\\
&=t^d\lambda_d(f)^{m+1}\bigl(\overline{H}^{d+1}\bigr)
\left(t-\frac{d+1}{\bigl(\overline{H}^{d+1}\bigr)}\frac{\widehat{\deg}_{d,\overline H}(f)}{\lambda_d(f)}\right)>0.
\end{align*}
Thus, both $t\overline H_{m,N}-\overline H_{m+1,N}$ and
$t\overline H_{m+1,N}-\overline H_{m,N}$ are arithmetically big.
Since $t$ is independent of $m$ and $N$, taking
$C_{f,\overline H}=t$ completes the proof.
\end{proof}

\begin{proof}[Proof of \cref{thm:arithmetic-temporal-estimate}]
Fix $1\leq k\leq d$ and $n\geq1$.
Set $N\coloneqq n+1$.
On $Y_N$, let $V_N$ denote the real span of $\Hbar_{0,N},\dots,\Hbar_{N,N}\in \widehat{\Pic}(Y_N)_{\mathrm{nef},\bQ}$.
Let $\cN_N\subset V_N$ (resp. $\cP_N\subset V_N$) be the closure of the cone generated by nef (resp. pseudo-effective) adelic $\bQ$-line bundles.
Then the tuple $(V_N,\cN_N,\cP_N,Q)$ is an ordered Hodge--Lorentz package of degree $d+1$ by \cref{lemma:arithmetic-realization}.
For $0\leq m\leq N$ and $0\leq i\leq d+1$, set
\[
\eta_m=\Hbar_{m,N},\qquad
\lambda=\lambda_d(f)>0,\qquad
\delta_i(m)=\widehat{\deg}_{i,\Hbar}(f^m)>0.
\]
The hypothesis \ref{thm:abstract-log-laplacian-two-time-shift} in \cref{thm:abstract-log-laplacian} follows from \cref{lemma:arithmetic-two-time-shift}, while \cref{lemma:arithmetic-adjacent-comparison} guarantees a constant $C_{f,\Hbar}>0$ such that for every $0\leq m<N$,
\[
C_{f,\Hbar}^{-1}\,\eta_m\preceq_{\cP_N}\eta_{m+1}\preceq_{\cP_N}C_{f,\Hbar}\,\eta_m.
\]
Hence the inequality \eqref{eq:arithmetic-temporal-estimate} follows from \cref{thm:abstract-log-laplacian} with $C'_{f,\Hbar,k}=2k(d+1-k)C_{f,\Hbar}^{2d+6}$.
\end{proof}

Once \cref{thm:arithmetic-temporal-estimate} has been established, the proofs of \cref{thm:arithmetic-degree-growth-upper-bound,thm:strict-break-asymptotic} are almost the same as that of \cite[Theorems~1.1 and 1.5]{Hu-degree-growth}.
We include them for completeness.

\begin{proof}[Proof of \cref{thm:strict-break-asymptotic}]
Fix $\eta$ with $\widehat{s}_k(f)<\eta<\widehat{\lambda}_k(f)$ and set $q=\eta/\widehat{\lambda}_k(f)$.
By \cref{cor:existence-lambda-hat}, one has
\[
\lim_{n\to\infty} \left(\frac{\sqrt{\widehat{\deg}_{k-1,\Hbar}(f^n)\,\widehat{\deg}_{k+1,\Hbar}(f^n)}}{\widehat{\deg}_{k,\Hbar}(f^n)}\right)^{1/n} = \frac{\widehat{s}_k(f)}{\widehat{\lambda}_k(f)} < \frac{\eta}{\widehat{\lambda}_k(f)} = q,
\]
and hence
\[
\frac{\sqrt{\widehat{\deg}_{k-1,\Hbar}(f^n)\,\widehat{\deg}_{k+1,\Hbar}(f^n)}}{\widehat{\deg}_{k,\Hbar}(f^n)} = O(q^n).
\]
Combining this with \cref{thm:arithmetic-temporal-estimate} yields that
\[
\left|\log\frac{\widehat{\deg}_{k,\Hbar}(f^n)^2}{\widehat{\deg}_{k,\Hbar}(f^{n-1})\,\widehat{\deg}_{k,\Hbar}(f^{n+1})}\right| = O(q^n).
\]
Finally, applying \cite[Lemma~4.1]{Hu-degree-growth} with $a_n=\widehat{\deg}_{k,\Hbar}(f^n)$ and $\alpha=\widehat{\lambda}_k(f)$, we get
\[
\widehat{\deg}_{k,\Hbar}(f^n) = C_{f,\Hbar,k}\, \widehat{\lambda}_k(f)^n \bigl(1+O(q^n)\bigr) = C_{f,\overline{H},k}\, \widehat{\lambda}_k(f)^n + O(\eta^n),
\]
which completes the proof of \eqref{eq:strict-break-asymptotic}.
\end{proof}

For an array $a=(a_{k,n})$, define its \textit{spatial} and \textit{temporal discrete Laplacian}, whenever the neighboring terms are defined, by
\begin{align*}
(\Delta_{\mathrm{sp}} a)_{k,n}&\coloneqq 2a_{k,n}-a_{k-1,n}-a_{k+1,n},\\
(\Delta_{\mathrm{tm}} a)_{k,n}&\coloneqq 2a_{k,n}-a_{k,n-1}-a_{k,n+1}.
\end{align*}

\begin{proof}[Proof of \cref{thm:arithmetic-degree-growth-upper-bound}]
The lower bound follows from \cref{cor:existence-lambda-hat}.
Fix a maximal affine interval $[r,s]$ of $(\log \widehat{\lambda}_i(f))_{i=0}^{d+1}$.
For $r\leq k\leq s$ and $n\geq0$, set
\[
u_{k,n} \coloneqq \log\frac{\widehat{\deg}_{k,\Hbar}(f^n)}{\widehat{\lambda}_k(f)^n}.
\]
By \cref{cor:existence-lambda-hat} again, we have
\[
\lim_{n\to\infty} \frac{u_{k,n}}{n} = 0.
\]
For any interior index $r<j<s$ (if any), affineness of $[r,s]$ yields $\widehat{\lambda}_j(f)^2 = \widehat{\lambda}_{j-1}(f)\widehat{\lambda}_{j+1}(f)$.
So, by \cref{thm:arithmetic-temporal-estimate}, there is a constant $C_0\coloneqq \max_{r<j<s} C'_{f,\Hbar,j}>0$ such that for every $n\geq1$,
\[
\left|(\Delta_{\mathrm{tm}}u)_{j,n}\right|
\leq C_0 \frac{\sqrt{\widehat{\deg}_{j-1,\Hbar}(f^n)\,\widehat{\deg}_{j+1,\Hbar}(f^n)}}{\widehat{\deg}_{j,\Hbar}(f^n)} = C_0 \exp\left(-\frac{1}{2}(\Delta_{\mathrm{sp}}u)_{j,n}\right).
\]
If $[r,s]\subseteq [1,d]$, then both $r$ and $s$ are strict break indices, so that $u_{r,n}$ and $u_{s,n}$ are bounded by \cref{thm:strict-break-asymptotic}.
When $r=0$ or $s=d+1$, we have $u_{0,n}=u_{d+1,n}=\log(\Hbar^{d+1})$, which is constant.
So, by applying \cite[Lemma~4.3]{Hu-degree-growth}, there is a constant $C'_{u,r,s}>0$ such that for every $r\leq k\leq s$ and $n\geq1$,
\[
u_{k,n}\leq C'_{u,r,s} + 2(k-r)(s-k) \log n.
\]
Thus the upper bound in \eqref{eq:arithmetic-degree-growth-bound} follows with $C_{f,\overline{H},r,s}\coloneqq e^{C'_{u,r,s}}$.
\end{proof}

\section{Proof of Theorem \ref{thm:ht-growth-upper-bound}}\label{sec:proof}

The goal of this section is to prove \cref{thm:ht-growth-upper-bound}.
Thanks to \cite[Theorem~1.1]{Hu-degree-growth} and its arithmetic analog established in \cref{thm:arithmetic-degree-growth-upper-bound}, it suffices to prove a uniform height estimate for the images $f^n(P)$ in terms of the geometric/arithmetic degrees of the maps.
We yet treat the number field case first, i.e., \cref{thm:ht-degree-bound-N}, which is reminiscent of \cite[Proof of Theorem~4.5]{Song26}.
The function field case is relatively easier since the corresponding height functions are defined geometrically; see \cref{thm:ht-growth-bound-F}.
After these preparations, we prove \cref{thm:ht-growth-upper-bound}.

\subsection{Number field case}\label{subsec:number-field}

Throughout this subsection, $X$ is a normal projective variety of dimension $d$ over a number field $K$.
We retain the notations in \S\ref{sec:arithmetic-prelim}.

\begin{theorem}\label{thm:ht-degree-bound-N}
Let $\overline{H}$ be a nef and big adelic line bundle on $X$ with $H$ ample.
Then there exist constants $C_{\overline{H}},C'_{\overline{H},h_H}>0$ such that for any dominant rational self-map $g\colon X\dashrightarrow X$ and any $P\in X_g(\overline{K})$,
\begin{equation}\label{eq:adelic-ht-degree-bound-N}
h_{\overline{H}}^+(g(P))\leq C_{\overline{H}}\,\widehat{\deg}_{1,\overline{H}}(g)h_{\overline{H}}^+(P),
\end{equation}
and
\begin{equation}\label{eq:ample-ht-degree-bound-N}
h_H^+(g(P))\leq C'_{\overline{H},h_H}\widehat{\deg}_{1,\overline{H}}(g)h_H^+(P).
\end{equation}
\end{theorem}

The following lemma, which allows us to choose sections with bounded coefficients, is crucial in the proof of \cref{thm:ht-degree-bound-N}.

\begin{lemma}\label{lem:subspace_cover}
Let $H$ be an ample line bundle on $X$.
Let $\ell_1,\ell_2,\dots,\ell_N\in H^0(X,H)$ be global sections, not all zero, and denote their base locus by $B\coloneqq \mathrm{Bs}(\ell_1,\dots,\ell_N)$. Fix a finite set $S\subset K$ with $\# S>(H^d)$ and an integer $1\leq m\leq d$. Then we can choose sections
\[
s_i=\sum_{j=1}^N b_{i,j}\ell_j,
\qquad b_{i,j}\in S,
\qquad 1\leq i\leq m,
\]
such that every irreducible component of $B$ of codimension $m$ in $X$ is an irreducible component of $\mathrm{Bs}(s_1,\dots,s_m)$.
\end{lemma}
\begin{proof}
If $B$ has no irreducible component of codimension $m$, there is nothing to prove. Otherwise, let $U$ be the complement in $X$ of the union of all irreducible components of $B$ of dimension at least $d-m+1$. Then $\dim B\cap U\leq d-m$ and every irreducible component of $B$ of codimension $m$ meets $U$.

We construct $s_1,\dots,s_m$ inductively such that
\[
Y_i\coloneqq \overline{U\cap\mathrm{Bs}(s_1,\dots,s_i)}^{\mathrm{Zar}}\subseteq X
\]
is pure of dimension $d-i$ and $\deg_H Y_i\leq (H^d)$. Since $Y_0=X$, the assertion holds for $i=0$. Suppose that $s_1,\dots,s_{i-1}$ have been chosen, and let $Z_1,\dots,Z_t$ be the irreducible components of $Y_{i-1}$. Since $H$ is ample,
\[
t\leq \sum_{k=1}^t\deg_H(Z_k)\leq \deg_H Y_{i-1}\leq (H^d).
\]
Moreover,
\[
\dim Z_k=d-i+1>d-m\geq\dim(B\cap U),
\]
so $Z_k$ is not contained in $B$ for all $k$. Thus
\[
W_k\coloneqq 
\left\{
(b_1,\dots,b_N)\in K^N:
Z_k\subseteq \mathrm{div}\left(\sum_{j=1}^N b_j\ell_j\right)
\right\}
\]
is a proper linear subspace of $K^N$. Since $\# S>(H^d)\geq t$, we have
\[\#\left(S^N\cap  \bigcup\limits_{k=1}^tW_k\right)\leq \sum_{k=1}^t \#\big(S^N\cap W_k\big)\leq t (\# S)^{N-1}<(\# S)^N,\]
which implies $S^N\not\subseteq \bigcup\limits_{k=1}^tW_k$. Thus we may choose $(b_{i,1},\dots,b_{i,N})\in S^N\setminus \bigcup\limits_{k=1}^tW_k$ and set
\[
s_i=\sum_{j=1}^N b_{i,j}\ell_j.
\]
By construction, $Z_k\not\subseteq \mathrm{div}(s_i)$ for all $1\leq k\leq t$. It follows that $Y_i$ is nonempty and pure of dimension $d-i$. Moreover, 
\[
\deg_H Y_i\leq\bigl(H^{d-i}\cdot H\cdot Y_{i-1}\bigr)= \deg_H Y_{i-1} \leq (H^d).
\]
This completes the induction.

Now let $V$ be an irreducible component of $B$ of codimension $m$. Since all $s_i$ vanish on $B$, we have $V\subseteq Y_m$. Note that $Y_m$ is pure of dimension $d-m$, then $V$ is an irreducible component of $\mathrm{Bs}(s_1,\dots,s_m)$.
\end{proof}

We now prove an adelic version of \cite[Lemma~4.4]{Song26}. A geometric version of this lemma appears in \cite[Lemma~3.3]{JSXZ24}.
\begin{lemma}\label{lem:nonproper_intersection}
Let $V\subset X$ be an irreducible subvariety of codimension $m$.
Let $\overline{L}_1,\dots,\overline{L}_m$ be nef adelic line bundles on $X$, and let
$s_i\in\widehat{H}^0(X,\overline{L}_i)$
be nonzero small sections for $1\leq i\leq m$. Suppose that $V$ is an irreducible component of
$\mathrm{Bs}(s_1,\dots,s_m)$. Then, for any nef adelic line bundles
$\overline{H}_1,\dots,\overline{H}_{d+1-m}$ on $X$, we have
\[
\bigl(\overline{H}_1\cdots\overline{H}_{d+1-m}\cdot\overline{L}_1\cdots\overline{L}_m\bigr)\geq
\bigl(\overline{H}_1\cdots\overline{H}_{d+1-m}\cdot V\bigr).
\]
\end{lemma}
\begin{proof}
We prove by induction on $m$. Suppose first that $m=1$.
By the induction formula in \cite[Théorème~4.1]{CLT09}
(see also \cite[Appendix, Arithmetic intersections]{YZ17}), we have
\[
\begin{aligned}
\big(\overline{H}_1\cdots\overline{H}_d\cdot \overline{L}_1\big)&= \big(\overline{H}_1\cdots\overline{H}_d\cdot \mathrm{div}(s_1)\big)-\sum_{v\in M_K}\int_{X_v^{\mathrm{an}}}\log\|s_1\|_vc_1(\overline{H}_1)_v\cdots c_1(\overline{H}_d)_v\\
&\geq \big(\overline{H}_1\cdots \overline{H}_{d}\cdot V\big).        
\end{aligned}
\]
The last inequality holds because $s_1$ is small, each $\overline{H}_i$ is nef, and $\mathrm{div}(s_1)$ contains $V$ with multiplicity at least one.

Now assume that the result holds for $m-1$ sections.
Since $V$ has codimension $m$, after relabeling,  there is an irreducible component
$W$ of $\mathrm{Bs}(s_1,\dots,s_{m-1})$
of codimension $m-1$ such that $W\not\subseteq\mathrm{div}(s_m)$ and $V$ is an irreducible component of $W\cap\mathrm{div}(s_m)$. Applying the induction hypothesis with the nef line bundles
$\overline{H}_1,\dots,\overline{H}_{d+1-m},\overline{L}_m$
and then using the induction formula, we obtain
\begin{align*}
&\,\big(\overline{H}_1\cdots\overline{H}_{d+1-m}\cdot\overline{L}_1\cdots\overline{L}_m\big) \geq \big(\overline{H}_1\cdots\overline{H}_{d+1-m}\cdot\overline{L}_m\cdot W\big)\\
=&\,\big(\overline{H}_1\cdots\overline{H}_{d+1-m}\cdot W\cdot \mathrm{div}(s_m)\big)-\sum_{v\in M_K}\int_{X_v^{\mathrm{an}}}\log\|s_m\|_vc_1(\overline{H}_1)_v\cdots c_1(\overline{H}_{d+1-m})_v\delta_{W_v}\\
\geq &\,\big(\overline{H}_1\cdots\overline{H}_{d+1-m}\cdot V\big).\qedhere
\end{align*}
\end{proof}

\begin{proof}[Proof of \cref{thm:ht-degree-bound-N}]
Fix an integer $r\geq 1$ such that $H_1\coloneqq H^{\otimes r}$ is very
ample, and set $\overline H_1\coloneqq \overline H^{\otimes r}$. Choose an embedding
$\iota\colon X\hookrightarrow\mathbb P_K^N$ with $\iota^*\mathcal O_{\mathbb P^N}(1)=H_1$. Let
$u_j\coloneqq \iota^*x_j\in H^0(X,H_1)$ for $0\leq j\leq N$, where $x_0,\ldots,x_N$
are the homogeneous coordinates on $\mathbb P_K^N$.

Now fix $P\in X(\overline K)\setminus I(g)$ and set $L=K(P)$.
We view $P$ as an $L$-rational point of $X_L$. Let $H_L$ and
$\overline H_L$ denote the pullbacks of $H$ and $\overline H$
to $X_L$, respectively, and set $H_2\coloneqq H_L^{\otimes r}, \overline H_2\coloneqq \overline H_L^{\otimes r}$. Write
\[
\iota(P)=[a_0:a_1:\cdots:a_N],
\qquad a_j\in L.
\]
Without loss of generality, we may assume that $a_0=1$. By abuse of notation, we also write $u_j$ for the pullback of $u_j$ to $X_L$. Set $\ell_i\coloneqq u_i-a_i u_0\in H^0(X_L,H_2)$ for $1\leq i\leq N$. Then $\mathrm{Bs}(\ell_1,\ldots,\ell_N)=\{P\}$.
By \cref{lem:subspace_cover}, we can choose sections
\[
s_i=\sum_{j=1}^N b_{i,j}\ell_j,
\qquad
b_{i,j}\in\{0,1,\ldots,C_1\},
\qquad 1\leq i\leq d,
\]
such that $P$ is a zero-dimensional irreducible component of
$\mathrm{Bs}(s_1,\ldots,s_d)$, where 
\[C_1=(H_2^d)_{X_L}=(H_1^d)_X=r^d(H^d)_X\in\bZ_{>0}.\]

\begin{claim}\label{claim:trivial-lb-choose}
There exist an auxiliary nef adelic line bundle $\overline E\in\widehat{\mathrm{Pic}}(L)$, and a constant $C_2\geq 0$ depending only on $X$ and $\overline{H}$, such that
\[
s_i\in
\widehat H^0
\bigl(X_L,\overline H_2+\pi^*\overline E\bigr),
\qquad 1\leq i\leq d,
\]
and
\[
\frac{\widehat{\deg}(\overline E)}{[L:\bQ]}
\leq 
r h_{\overline H}(P)+C_2
\leq
(C_2+r)h_{\overline H}^+(P),
\]
where $\pi:X_L\to\operatorname{Spec}L$ is the structural morphism.    
\end{claim}

\medskip
\noindent\textit{Proof of \cref{claim:trivial-lb-choose}.}
Let $w$ be a place of $L$ above $v\in M_K$. Define
\[
U_v\coloneqq 
\max_{0\leq j\leq N}
\|u_j\|_{\overline H_1,v,\mathrm{sup}}>0,\quad \text{and}\quad U_w\coloneqq 
\max_{0\leq j\leq N}
\|u_j\|_{\overline H_2,w,\mathrm{sup}}>0.
\]
Then $U_v\leq 1$ for all but finitely many $v$ and $U_w\leq 1$ for all but finitely many $w$. By the definition
of the pullback metrics, for $0\leq j\leq N$, we have
\[
U_w=\max_{0\leq j\leq N}\|u_j\|_{\overline H_2,w,\mathrm{sup}}
\leq
\max_{0\leq j\leq N}\|u_j\|_{\overline H_1,v,\mathrm{sup}}^{[L_w:K_v]}
= U_v^{[L_w:K_v]}.
\]
Since $u_j(P)=a_j u_0(P)$, it follows that
\[
|a_j|_w\,\|u_0(P)\|_{\overline H_2,w}
=\|u_j(P)\|_{\overline H_2,w}
\leq U_w.
\]

If $w$ is a finite place, we have
\[
\|\ell_j\|_{\overline H_2,w,\mathrm{sup}}\,
 \|u_0(P)\|_{\overline H_2,w}
\leq
\max\bigl\{
\|u_j\|_{\overline H_2,w,\mathrm{sup}},
|a_j|_w\|u_0\|_{\overline H_2,w,\mathrm{sup}}
\bigr\}
\|u_0(P)\|_{\overline H_2,w}
\leq U_w^2.
\]
Since $|b_{i,j}|_w\leq 1$, we obtain
\[
\|s_i\|_{\overline H_2,w,\mathrm{sup}}
\leq
\max_{1\leq j\leq N}
\|\ell_j\|_{\overline H_2,w,\mathrm{sup}}
\leq
\frac{U_w^{2}}{\|u_0(P)\|_{\overline H_2,w}}.
\]
If $w$ is an infinite place, we have
\[\|\ell_j\|_{\overline H_2,w,\mathrm{sup}}\,
 \|u_0(P)\|_{\overline H_2,w}
\leq
2\bigl(
\|u_j\|_{\overline H_2,w,\mathrm{sup}}+|a_j|_w\|u_0\|_{\overline H_2,w,\mathrm{sup}}\bigr)\|u_0(P)\|_{\overline H_2,w}
\leq 4U_w^{2}.\]
Since $|b_{i,j}|_w\leq C_1^2$, we obtain
\[
\|s_i\|_{\overline H_2,w,\mathrm{sup}}\leq N\sum_{j=1}^N|b_{i,j}|_w
\|\ell_j\|_{\overline H_2,w,\mathrm{sup}}\leq\frac{4N^2C_1^2U_w^2}{\|u_0(P)\|_{\overline H_2,w}}.
\]
Define
\[
 c_w\coloneqq \begin{cases}
\max\{1,U_w^2\},&w\in M_{L,f},\\
\max\{1,4N^2C_1^2U_w^2\},&w\in M_{L,\infty},
\end{cases}\quad
c_v\coloneqq \begin{cases}
\max\{1,U_v^2\},&v\in M_{K,f},\\
\max\{1,4N^2C_1^2U_v^2\},&v\in M_{K,\infty}.
\end{cases}
\]
Then $\log c_w\leq [L_w:K_v]\log c_v$ for all $w\mid v$. Define 
\[C_2\coloneqq \dfrac{1}{[K:\bQ]}\sum\limits_{v\in M_K} \log c_v.\] Since $c_v=1$ for all but finitely many $v$, the constant $C_2\geq 0$ is finite.

Define $\overline E$ to be the trivial line bundle on
$\operatorname{Spec}L$ equipped with the metrics
\[
\|1\|_{\overline E,w}\coloneqq \frac{\|u_0(P)\|_{\overline H_2,w}}{c_w}.
\]
 Then $\|1\|_{\overline E,w}\leq 1$ for all $w$ and the equality holds for all but finitely many $w$. In particular, $\overline E$ is nef. By the construction,
\[
\|s_i\|_{\overline H_2+\pi^*\overline E,w,\mathrm{sup}}=
\|s_i\|_{\overline H_2,w,\mathrm{sup}}\,
\frac{\|u_0(P)\|_{\overline H_2,w}}{c_w}\leq 1.
\]
Thus each $s_i$ is a small section of $\overline H_2+\pi^*\overline E$ for all $i$.

We compute
\[
\begin{aligned}
\frac{\widehat{\deg}(\overline E)}{[L:\bQ]}
&=
\frac{1}{[L:\bQ]}\sum_{v\in M_K}\sum_{w\mid v}\left(
\log c_w-\log\|u_0(P)\|_{\overline H_2,w}
\right)\\
&\leq 
\sum_{v\in M_K}\sum_{w\mid v}
\frac{[L_w:K_v]}{[L:\bQ]}\log c_v+h_{\overline H_2}(P)
\\
&=
C_2+r h_{\overline H}(P)\\
&\leq
(C_2+r)h_{\overline H}^+(P).
\end{aligned}
\]
This proves the claim.\bqed

We now return to the proof of \cref{thm:ht-degree-bound-N}.
Let $\Gamma\subset X\times_K X$ be the closure of the graph of
$g$, let $\widetilde\Gamma$ be its normalization, and denote
the two projections by $\pi_1,\pi_2\colon \widetilde\Gamma\longrightarrow X$. Write $\pi_{1,L},\pi_{2,L}\colon \widetilde\Gamma_L\to X_L$ for their base changes to $L$. Then we have 
\[
\widehat{\deg}_{1,\overline H_L}(g_L)=[L:K]\,\widehat{\deg}_{1,\overline H}(g),\quad \deg_{1,H_L}(g_L)=\deg_{1,H}(g).
\]
Since $P\notin I(g)$, $\pi_{1}$ is an
isomorphism over a neighborhood of $P$. Thus
$P'\coloneqq \pi_{1}^{-1}(P)$ is a point of
$\widetilde\Gamma(L)$ satisfying $\pi_{2}(P')=g(P)$.
Moreover, viewing $P'$ as an $L$-rational point of $\widetilde{\Gamma}_L$, it is a zero-dimensional irreducible component of the base locus $\mathrm{Bs}(\pi_{1,L}^*s_1,\dots,\pi_{1,L}^*s_d)$.
By the definition \eqref{eq:adelic-Weil-ht} of $h_{\overline H}$ and applying \cref{lem:nonproper_intersection} with $\Lbar_i=\pi_{1,L}^*(\overline H_2+\pi^*\overline E)$ for $1\leq i\leq d$, we obtain
\begin{align*}
h_{\overline H}(g(P))&=
\frac{\bigl(\pi_{2,L}^*\overline H_L\cdot P'\bigr)_{\widetilde\Gamma_L}}{[L:\bQ]}
\leq 
\frac{
\bigl(
\pi_{2,L}^*\overline H_L
\cdot
\bigl(\pi_{1,L}^*(\overline H_2+\pi^*\overline E)\bigr)^d
\bigr)_{\widetilde\Gamma_L}
}{[L:\bQ]}\\
&=
\frac{r^d}{[L:\bQ]}
\bigl(
\pi_{2,L}^*\overline H_L
\cdot(\pi_{1,L}^*\overline H_L)^d
\bigr)_{\widetilde\Gamma_L}+
\frac{dr^{d-1}}{[L:\bQ]}
\bigl(
\pi_{2,L}^*\overline H_L
\cdot(\pi_{1,L}^*\overline H_L)^{d-1}
\cdot\pi_{1,L}^*\pi^*\overline E
\bigr)_{\widetilde\Gamma_L}\\
&= \frac{r^d}{[K:\bQ]}\,\widehat{\deg}_{1,\overline H}(g)+\frac{dr^{d-1}\,\widehat{\deg}(\overline E)}{[L:\bQ]}\deg_{1,H}(g)\\
&\leq \frac{r^d}{[K:\bQ]}\,\widehat{\deg}_{1,\overline H}(g) h_{\overline H}^+(P) + dr^{d-1} (C_2+r) \deg_{1,H}(g) h_{\overline H}^+(P),
\end{align*}
where the last inequality follows from \cref{claim:trivial-lb-choose} and the convention $h_{\overline H}^+=\max\{1,h_{\overline H}\}$.
Recall that by \cref{lemma:degree-comparison}, there exists a constant $C_3>0$, depending only on $X$ and $\overline H$, such that
\[
\widehat{\deg}_{1,\overline H}(g) \geq C_3 \deg_{1,H}(g).
\]
Combining the last two displayed inequalities together yields that
\[
h_{\overline H}(g(P)) \leq C_{\overline{H}}\,\widehat{\deg}_{1,\overline H}(g)h_{\overline H}^+(P),
\]
where
\[
C_{\overline{H}}\coloneqq \frac{r^d}{[K:\bQ]}+\frac{dr^{d-1}(C_2+r)}{C_3}
\]
depends only on $X$ and $\overline H$.
Finally, since $\deg_{1,H}(g)\geq 1$, we have
\[
C_{\overline{H}}\,\widehat{\deg}_{1,\overline H}(g)h_{\overline H}^+(P)
\geq C_{\overline{H}}\,C_3\deg_{1,H}(g)
\geq dr^{d-1}(C_2+r)\geq 1.
\]
Hence
\[
h_{\overline H}^+(g(P))
\leq
C_{\overline{H}}\,\widehat{\deg}_{1,\overline H}(g)h_{\overline H}^+(P),
\]
which proves the uniform height estimate \eqref{eq:adelic-ht-degree-bound-N} in \cref{thm:ht-degree-bound-N}.
The inequality~\eqref{eq:ample-ht-degree-bound-N} follows from the fact that
$h_H^+ \asymp h_{\Hbar}^+$,
where the implied constant depends only on $X$, $h_H$, and $\overline{H}$.
\end{proof}

\subsection{Function field case}\label{subsec:function-field}

Throughout this subsection, let $B$ be a smooth projective curve over an algebraically closed field $\bk$ of arbitrary characteristic, and let $K\coloneqq\bk(B)$ be its function field. Let $X$ be a normal projective variety of dimension $d$ over $K$, and let $H$ be an ample line bundle on $X$.
We aim to prove a function-field analog of \cref{thm:ht-degree-bound-N}.

Choose a positive integer $e$ and a model
$(\mathcal{X},\mathcal{H})$ of $(X,H^{\otimes e})$ over $B$. Here, $\mathcal{X}$ is normal and projective over $B$, with generic
fiber $\mathcal{X}_K\simeq X$, and $\mathcal{H}$ is an ample line
bundle on $\mathcal{X}$ satisfying
$\mathcal{H}_K\simeq H^{\otimes e}$.
Such a model exists for sufficiently large $e$.

For each $P\in X(\overline{K})$, let
\[
    V_P\coloneqq \overline{\{P\}}^{\mathrm{Zar}}\subset\mathcal{X}
\]
be the corresponding multisection, and define
\[
    h_{\mathcal{H}}(P)
    \coloneqq \frac{\mathcal{H}\cdot V_P}{[K(P):K]}.
\]
Then $e^{-1}h_{\mathcal{H}}$ is a Weil height associated with $H$.
Every dominant rational self-map $g\colon X\dashrightarrow X$ defined
over $K$ extends to a dominant rational self-map
$G\colon \mathcal{X}\dashrightarrow\mathcal{X}$ over $B$.

\begin{theorem}\label{thm:ht-growth-bound-F}
With notations as above, there exist constants $C_{\mathcal{X},\mathcal{H}},C'_{\mathcal{X},\mathcal{H},h_H}\geq1$ such that for any dominant rational self-map $g\colon X\dashrightarrow X$ and any $P\in X_g(\overline{K})$,
\begin{equation}\label{eq:model-ht-degree-bound-F}
h_{\mathcal{H}}(g(P))\leq C_{\mathcal{X},\mathcal{H}}\deg_{1,\mathcal{H}}(G)h_{\mathcal{H}}(P),
\end{equation}
and
\begin{equation}\label{eq:ample-ht-degree-bound-F}
h_H^+(g(P))\leq C'_{\cX,\cH,h_H}\deg_{1,\cH}(G)h_H^+(P).
\end{equation}
\end{theorem}

\begin{lemma}\label{lem:section_choose}
Let $Y$ be a normal projective variety of dimension $d$ over an
algebraically closed field $\bk$, and let $L$ be a very ample line
bundle on $Y$. There exists a constant $C>0$ depending only on $d$ such that, for every
irreducible subvariety $V\subset Y$ of codimension $r$, there are
positive integers $n_1,\ldots,n_r$ and sections
$s_i\in H^0(Y,L^{\otimes n_i})$ such that $V$ is an irreducible
component of $\mathrm{Bs}(s_1,\ldots,s_r)$ and
\[
    n_1\cdots n_r\leq C\deg_L V.
\]
\end{lemma}

\begin{proof}
Choose a finite surjective morphism $p\colon Y\to\bP^d$ such that
$L\simeq p^*\mathcal{O}_{\bP^d}(1)$. Set $W=p(V)$. Since $p$ is finite, $W$ is an irreducible subvariety
of $\bP^d$ of codimension $r$. Let $e\coloneqq [\bk(V):\bk(W)]$. Since $p_*[V]=e[W]$, the projection
formula gives
\[
\begin{aligned}
    \deg_L V=L^{d-r}\cdot V=\mathcal{O}_{\bP^d}(1)^{d-r}\cdot p_*[V]=e\deg_{\mathcal{O}_{\bP^d}(1)}W\geq\deg_{\mathcal{O}_{\bP^d}(1)}W.
\end{aligned}
\]

By \cite[Theorem 1.3]{Walsh20}, there exist $n_1,\ldots,n_r\in\bZ_{>0}$ and sections
$\ell_i\in H^0(\bP^d,\mathcal{O}_{\bP^d}(n_i))$ such that
$W$ is an irreducible component of $\mathrm{Bs}(\ell_1,\ldots,\ell_r)$ and
\[
    n_1\cdots n_r\leq C\deg_{\mathcal{O}_{\bP^d}(1)}W,
\]
where $C$ depends only on $d$.

Set $s_i=p^*\ell_i\in H^0(Y,L^{\otimes n_i})$. It suffices to prove $V$ is an irreducible component of $ \mathrm{Bs}(s_1,\ldots,s_r)=p^{-1}\bigl(\mathrm{Bs}(\ell_1,\ldots,\ell_r)\bigr)$. Suppose $V'\subset \mathrm{Bs}(s_1,\ldots,s_r)$ is an irreducible component containing $V$. Then $p(V')$ is an irreducible component of
$\mathrm{Bs}(\ell_1,\ldots,\ell_r)$ containing $W$. So $p(V')=W$, which implies $\dim V'=\dim W=\dim V$. Hence $V'=V$, which completes the proof.
\end{proof}

\begin{proof}[Proof of \cref{thm:ht-growth-bound-F}]
Choose a positive integer $r$ such that $\mathcal{H}_1=\mathcal{H}^{\otimes r}$ is a very ample line bundle on $\mathcal{X}$.
Fix $P\in X_g(\overline{K})$.
By \cref{lem:section_choose}, there exist a constant $C_1\geq1$, positive integers
$n_1,\ldots,n_d$, and sections
$s_i\in H^0(\mathcal{X},\mathcal{H}_1^{\otimes n_i})$ such that
$V_P$ is an irreducible component of
$\mathrm{Bs}(s_1,\ldots,s_d)$ and
\[
    n_1\cdots n_d\leq C_1\deg_{\mathcal{H}_1}V_P=rC_1\deg_{\mathcal{H}}V_P.
\]
Note that $C_1\geq1$ depends only on $d$.

Let $\Gamma\subset\mathcal{X}\times_{\bk}\mathcal{X}$ be the closure
of the graph of $G$, let $\widetilde{\Gamma}$ be its normalization,
and denote the natural projections by
\[
    \pi_1,\pi_2\colon \widetilde{\Gamma}\longrightarrow\mathcal{X}.
\]
Since $P\notin I(g)$, the morphism $\pi_1$ is an isomorphism over
a neighborhood of the generic point $P$ of $V_P$.
Consequently, the strict transform $\pi_1^{\#}V_P$ is an
irreducible component of
$\mathrm{Bs}(\pi_1^*s_1,\ldots,\pi_1^*s_d)$.
Moreover,
\[
    \pi_{2,*}[\pi_1^{\#}V_P]
    =[K(P):K(g(P))]\,[V_{g(P)}].
\]
By the projection formula and \cite[Lemma 3.3]{JSXZ24}, we obtain
\[
\begin{aligned}
    h_{\mathcal{H}}(g(P))
    &=\frac{\mathcal{H}\cdot V_{g(P)}}{[K(g(P)):K]}=\frac{\mathcal{H}\cdot\pi_{2,*}\pi_1^{\#}V_P}
            {[K(P):K]}=\frac{\pi_2^*\mathcal{H}\cdot \pi_1^{\#}V_P}
            {[K(P):K]}\\
    &\leq(n_1\cdots n_d)
        \frac{\pi_2^*\mathcal{H}\cdot(\pi_1^*\mathcal{H}_1)^d}
             {[K(P):K]}\\
    &\leq C_1r^{d}\deg_{1,\mathcal{H}}(G)
        \frac{\deg_{\mathcal{H}_1}V_P}{[K(P):K]}\\
    &=C_1r^{d+1}\deg_{1,\mathcal{H}}(G)\,h_{\mathcal{H}}(P).
\end{aligned}
\]
Taking $C_{\mathcal{X},\mathcal{H}}\coloneqq C_1r^{d+1}$ proves \eqref{eq:model-ht-degree-bound-F}, while the inequality~\eqref{eq:ample-ht-degree-bound-F} follows from the standard comparison of ample heights $h_H^+ \asymp h_{\mathcal{H}}^+$.
\end{proof}

\begin{lemma}\label{lem:maximal_affine_interval-F}
With the notation above, let $f\colon X\dashrightarrow X$ be a dominant
rational self-map extending to a dominant rational self-map
$F\colon \mathcal{X}\dashrightarrow\mathcal{X}$.
\begin{enumerate}[label=\textnormal{(\arabic*)}]
\item If $\lambda_1(f)>1$, then an interval $[0,s]$ is a maximal
affine interval of $(\log\lambda_i(f))_{i=0}^d$ if and only if
it is a maximal affine interval of
$(\log\lambda_i(F))_{i=0}^{d+1}$.

\item If $\lambda_1(f)=1$, then $\lambda_i(F)=1$ for all
$0\leq i\leq d+1$. In particular, $[0,d+1]$ is the maximal affine
interval of $(\log\lambda_i(F))_{i=0}^{d+1}$.
\end{enumerate}
\end{lemma}

\begin{proof}
By the relative dynamical degree formula (see \cite[Theorem~1.3]{Truong20}; also known as the mixed degree formula in \cite[Theorem~B]{MX26}), we have
\[
    \lambda_i(F)
    =\max\{\lambda_i(f),\lambda_{i-1}(f)\},
    \qquad 1\leq i\leq d.
\]
Moreover, $\lambda_0(F)=1$ and $\lambda_{d+1}(F)=\deg(F)=\deg(f)=\lambda_d(f)$. The result follows by the same argument as in \cref{lem:maximal_affine_interval-N}.
\end{proof}

\subsection{Proof of Theorem \ref{thm:ht-growth-upper-bound}}

\begin{proof}[Proof of \cref{thm:ht-growth-upper-bound}]
In \cref{setting:number-field}, let $\overline{H}\in \widehat{\mathrm{Pic}}(X)$ be a nef and big adelic line bundle extending the ample line bundle $H$ on $X$.
In \cref{setting:function-field}, let $(\mathcal{X},\mathcal{H})$ be a model of $(X,H^{\otimes e})$ over $B$ for some $e\gg1$ and let $F\colon \mathcal{X}\dashrightarrow\mathcal{X}$ denote the dominant rational self-map of $\cX$ extending $f$.
Then by applying \cref{thm:ht-degree-bound-N,thm:ht-growth-bound-F} to $g=f^n$, respectively, there is a constant $C_1\geq1$ such that for every $n\geq1$ and $P\in X_f(\Kbar)$,
\[
h_H^+(f^n(P))\leq C_1a_nh_H^+(P),
\]
where
\[
a_n=
\begin{cases}\widehat{\deg}_{1,\Hbar}(f^n), & \text{in \cref{setting:number-field}},\\
\deg_{1,\cH}(F^n), & \text{in \cref{setting:function-field}}.
\end{cases}
\]

Let $[0,s]$ be the maximal affine interval of $(\log\widehat{\lambda}_i(f))_{i=0}^{d+1}$ or $(\log\lambda_i(F))_{i=0}^{d+1}$, respectively.
Then by \cref{thm:arithmetic-degree-growth-upper-bound} and \cite[Theorem~1.1]{Hu-degree-growth}, there is a constant $C_2>0$, such that
\[
a_n \leq C_2 n^{2(s-1)}
\begin{cases}\widehat{\lambda}_{1}(f)^n, & \text{in \cref{setting:number-field}},\\
\lambda_1(F)^n, & \text{in \cref{setting:function-field}}.
\end{cases}
\]
Finally, by \cref{lem:maximal_affine_interval-N} and \cref{lem:maximal_affine_interval-F}, respectively, we know that if $\lambda_1(f)=1$ then $s=d+1$, and both $\widehat{\lambda}_{1}(f)$ and $\lambda_1(F)$ are equal to $1$.
If $\lambda_1(f)>1$, the same lemmas show that $\widehat{\lambda}_{1}(f)=\lambda_1(f)$, $\lambda_1(F)=\lambda_1(f)$ in the respective settings, and that $[0,s]$ is also the maximal affine interval of $(\log\lambda_i(f))_{i=0}^d$.
In either case, taking $C_{f,h_H}\coloneqq C_1C_2$ concludes the proof of \cref{thm:ht-growth-upper-bound}.
\end{proof}

\begin{remark}\label{rmk:jump-explanation}
The proof of \cref{thm:ht-growth-upper-bound} explains the different
polynomial exponents in \eqref{eq:HI-ab-var-case}.
In both settings, the exponent is $2(s-1)$, where $[0,s]$ is the initial maximal affine interval of
\[
(\log\widehat{\lambda}_i(f))_{i=0}^{d+1}
\quad\text{or}\quad
(\log\lambda_i(F))_{i=0}^{d+1},
\]
in the number-field and function-field settings, respectively.
If $\lambda_1(f)>1$, this interval coincides with the initial maximal affine interval of $(\log\lambda_i(f))_{i=0}^{d}$, so $s\leq d$.
If $\lambda_1(f)=1$, the extended sequence is constant, so $s=d+1$.
\end{remark}

\bibliographystyle{amsalpha}
\bibliography{mybib}

\end{document}